\documentclass[12pt]{amsart}
\newtheorem{thm}{Theorem}[section]
\newtheorem{lem}[thm]{Lemma}
\newtheorem{cor}[thm]{Corollary}
\newtheorem{prop}[thm]{Proposition}
\theoremstyle{remark}
\newtheorem{defn}[thm]{Definition}
\newtheorem{rem}[thm]{Remark}
\newtheorem{exa}[thm]{Example}
\newtheorem{notation}[thm]{Notation}

\newtheorem*{prfofthm}{Proof of Theorem 2.1}
\newtheorem*{acknowledgement}{Acknowledgment}
\makeatletter
 
 \@addtoreset{equation}{section}

\makeatother

\title{Slopes of Fibered Surfaces with a Finite Cyclic Automorphism}
\author{Makoto Enokizono}
\address{Makoto Enokizono,
	Department of Mathematics,
	Graduate School of Science,
	Osaka University,
	Toyonaka, Osaka 560-0043, Japan}
\email{m-enokizono@cr.math.sci.osaka-u.ac.jp}
\usepackage[top=35truemm,bottom=35truemm,left=25truemm,right=25truemm]{geometry}
\usepackage{amsmath,cases}
\usepackage{amsfonts}
\usepackage[all]{xy}
\usepackage{here}
\usepackage{latexsym}
 \def\qed{\hfill $\Box$} 
\begin{document}
\maketitle
\begin{abstract}
We study slopes of finite cyclic covering fibrations of a fibered surface. We give the best possible lower bound of the slope of these fibrations. We also give the slope equality of finite cyclic covering fibrations of a ruled surface and observe the local concentration of the global signature of these surfaces on a finite number of fiber germs. We also give an upper bound of the slope of finite cyclic covering fibrations of a ruled surface.
\end{abstract}

\section*{Introduction}

Let $f\colon S\rightarrow B$ be a surjective morphism from a complex smooth projective surface $S$ to a smooth projective curve $B$ with connected fibers. 
The datum $(S,f,B)$ or simply $f$ is called a {\em fibered surface} or a fibration. 
A fibered surface $f$ is said to be {\em relatively minimal} if there exist no $(-1)$-curves contained in fibers of $f$, where a $(-1)$-curve is a non-singular rational curve with self-intersection number $-1$. 
The genus $g$ of a fibered surface $f$ is defined to be that of a general fiber of $f$. 
We put $K_f=K_S-f^*K_B$ and call it the relative canonical bundle. 

Assume that $f\colon S\to B$ is a relatively minimal fibration of genus $g\geq 2$, and 
consider the following three relative invariants:
\begin{align*}
\chi_f:=&\ \chi(\mathcal{O}_S)-(g-1)(b-1),\\
K_f^2=&\ K_S^2-8(g-1)(b-1),\\
e_f:=&\ e(S)-4(g-1)(b-1),
\end{align*}
where $b$ and $e(S)$ respectively denote the genus of the base curve $B$ and the topological Euler characteristic of $S$.
Then the following are well known:
\begin{itemize}
\item (Noether) $12\chi_f=K_f^2+e_f$.
\item (Arakelov) $K_f$ is nef.
\item (Ueno) $\chi_f\ge 0$, and $\chi_f=0$ if and only if $f$ is locally trivial (i.e., a holomorphic fiber bundle).
\item (Segre) $e_f\ge 0$, and $e_f=0$ if and only if $f$ is smooth.
\end{itemize}

When $f$ is not locally trivial, we put 
$$
\lambda_f=\frac{K_f^2}{\chi_f}
$$
and call it the {\em slope} of $f$.
Then one sees $0<\lambda_f\le 12$ from the above results. 
The slope of a fibration has proven to be sensible to a lot of geometric properties, both of the fibers of $f$ and of the surface $S$ itself (cf.\ \cite{ak}). 
A fibration of slope $12$ is called a Kodaira fibration, first examples of which were constructed by Kodaira in \cite{kod}.
In particular, the upper bound is sharp.
As to the lower bound, Xiao showed in \cite{xiao2} the inequality
$$
\lambda_f\ge 4-\frac{4}{g},
$$
which is nowadays called the {\em slope inequality}.
Furthermore, fibrations with slope $4-4/g$ are turned out to be of hyperelliptic type (\cite{kon} and \cite{xiao2}).
Hence a shaper lower bound is expected for non-hyperelliptic fibrations.

It is generally believed that there is the lower bound of the slope 
depending on the gonality (or the Clifford index) of fibrations.
Though there are several attempts, a general bound is still in fancy.
To attack such a problem, one of the most hopeful strategies may be to extend a 
special kind of linear system (e.g., a gonality pencil) or an automorphism of a general fiber to the whole surface 
and to study the fibration through the covering structure on $S$ thus obtained.
However, it is usually impossible to have the desired extension, mainly because the object in question 
on the fiber is not unique and the monodromy forces it to change from one to 
another (see, \cite{barja-naranjo}).
So, we are naturally led to consider the toy case that the fibration $f:S\to B$ is 
obtained from another fibration $W\to B$ via a covering map $S\to W$.
Besides the hyperelliptic fibrations, one of the remarkable results in this direction obtained so far is due to Cornalba and Stoppino 
\cite{dbl}.
They gave the lower bound for the slope of double covering fibrations and constructed examples showing its sharpness, 
extending a former result for bi-elliptic fibrations by Barja \cite{ba}.
See also \cite{luzuo} for recent developments of double covering fibrations.

In this paper, we consider fibered surfaces induced from a particular type of 
cyclic coverings, and try to generalize results for double coverings. 
More precisely, we shall work in the following situation.
A relatively minimal fibration $f\colon S\to B$ of genus $g\ge 2$ is 
called a {\em primitive cyclic covering fibration} of type $(g,h,n)$, if there exist a (not necessarily relatively minimal) fibration $\widetilde{\varphi}\colon\widetilde{W}\to B$ of genus $h\ge 0$, and a finite cyclic covering
$$
\widetilde{\theta}\colon\widetilde{S}=
\mathrm{Spec}_{\widetilde{W}}\left(\bigoplus_{j=0}^{n-1}
\mathcal{O}_{\widetilde{W}}(-j\widetilde{\mathfrak{d}})\right)\to \widetilde{W}
$$
of order $n$ branched along a smooth curve $\widetilde{R}\in |n\widetilde{\mathfrak{d}}|$
for some $n\geq 2$ and $\widetilde{\mathfrak{d}}\in 
\mathrm{Pic}(\widetilde{W})$
such that $f$ is the relatively minimal 
model of $\widetilde{f}:=\widetilde{\varphi}\circ \widetilde{\theta}$.
Here, we put the adjective ``primitive'', because a finite cyclic covering 
between non-singular surfaces is not necessarily obtained in this way. 
Nevertheless, the assumption would be acceptable, because the situation in \cite{dbl} is exactly the case $n=2$ and $h>0$, and a hyperelliptic fibration is nothing more than a primitive cyclic covering fibration of type $(g,0,2)$. 
Furthermore, Kodaira fibrations in \cite{kod} are among primitive cyclic covering fibrations with $\widetilde{W}$ being product of two curves.

As the first main result, we shall show the following:

\begin{thm}
Let $f\colon S\rightarrow B$ be a primitive cyclic covering fibration of type $(g,h,n)$. 
If $h\ge 1$ and $g\ge (2n-1)(2hn+n-1)/(n+1)$, then we have 
\begin{equation*}
K_f^2\ge\frac{24(g-1)(n-1)}{2(2n-1)(g-1)-n(n+1)(h-1)}\chi_f.
\end{equation*}
\end{thm}

\noindent 
Moreover, we will construct an example showing that the inequality is sharp (Example~\ref{sharpex}). 
Putting $n=2$, one recovers Cornalba-Stoppino's inequality shown in \cite{dbl}. 
The essential idea of the proof is analogous to the double covering case. 
We take a relatively minimal model $\varphi\colon W\to B$ of $\widetilde{\varphi}\colon \widetilde{W}\to B$ (unique when $h>0$) and consider the effects to the invariants of the ``canonical resolution'' of singular points of the branch curve $R\subset W$ (obtained as the direct image of $\widetilde{R}$). 
One of the striking facts in our setting is that the multiplicities of 
singular points must be either $0$ or $1$ modulo $n$ (Lemma \ref{multlem}).

When $h=0$, we can move from one relatively minimal model $\varphi\colon W\to B$ to another via elementary transformations among $\mathbb{P}^1$-bundles, in order to standardize the branch locus.
This enables us to prove a more accurate result. 
Namely, we shall show the {\em slope equality} for them.

\begin{thm}
There exists a function $\mathrm{Ind}\colon 
 \mathcal{A}_{g,0,n}\to \mathbb{Q}_{\geq 0}$ from the set 
 $\mathcal{A}_{g,0,n}$ of all fiber germs of primitive cyclic covering 
 fibrations of type $(g,0,n)$ such that $\mathrm{Ind}(F_p)=0$ for a 
 general $p\in B$ and
\begin{equation*}
K_f^2=\frac{24(g-1)(n-1)}{2(2n-1)(g-1)+n(n+1)}\chi_f+\sum_{p\in B}{\rm Ind}(F_p)
\end{equation*}
for any primitive cyclic covering fibration $f\colon S\rightarrow 
 B$ of type $(g,0,n)$.
\end{thm}

\noindent
This is a generalization of the hyperelliptic case (cf. \cite{pi1}). 
$\mathrm{Ind}(F_p)$ is called the {\em Horikawa index} of the fiber $F_p$ and defined in terms of {\em singularity indices} (Definition \ref{sinddef}) over $p\in B$ of the branch locus (see Theorem \ref{slopeeq} for the detail). 
This enables us to observe that the signature of $S$ is concentrated on singular fibers 
by introducing the {\em local signature} of a fiber (Corollary~\ref{signcor}). 
For a general discussion on the slope equalities, Horikawa index and the local signature, see \cite{ak}. 

As the third result, we will give an upper bound of the slope.

\begin{thm}
Let $f\colon S\to B$ be a primitive cyclic covering fibration of type $(g,0,n)$ and assume $n\ge 4$.
Put $r:=\displaystyle{\frac{2g}{n-1}+2}$,
$\delta:=\left\{\begin{array}{l}
0, \ \text{if}\ r\in 2n\mathbb{Z}, \\
1, \ \text{if}\ r\not\in 2n\mathbb{Z}. \\
\end{array}
\right.$
Then, the following hold:

\smallskip

\noindent
$(1)$ If $n\le r< n(n-1)$, then
$$
K_f^2\leq \left(12-\frac{48n^2(r-1)}{(n-1)(n+1)(r^2-\delta n^2)}\right)\chi_f.
$$

\smallskip

\noindent
$(2)$ If $r\ge n(n-1)$, then
$$
K_f^2\leq \left(12-\frac{48n(n-1)(r-1)}{n(n+1)r^2-8(2n-1)r+24n-\delta n^3(n+1)}\right)\chi_f.
$$
\end{thm}

\noindent
This in particular implies that we cannot have a Kodaira fibration when $h=0$ and $n\ge 4$, even if we consider not only Kodaira's original constructions but also such extended category of primitive cyclic covering fibrations. Unfortunately, it seems that our method does not work well when $n=3$, and we fail to obtain an upper bound. We leave it as a future study. An upper bound for the slope of hyperelliptic fibrations was obtained by Matsusaka in \cite{ma} which was improved by Xiao \cite{xiao_book}.

\begin{acknowledgement}\normalfont
The author expresses his sincere gratitude to Professor Kazuhiro Konno for his valuable advice and warm encouragement.
\end{acknowledgement}

\section{Primitive cyclic covering fibrations}
Throughout the paper, $n$ denotes an integer greater than $1$.
We begin with the following elementary lemma. Note that it does not hold for $n=3$, as the case $a=b=1$ shows.

\begin{lem} \label{easylem}
Let $n$ be a positive integer and $a$, $b$ integers such that ${\rm gcd}(a,b,n)=1$. If $n\ge 4$, then either $a+2b\notin n\mathbb{Z}$ or $2a+b\notin n\mathbb{Z}$.
\end{lem}

\begin{proof}
Suppose on the contrary $a+2b\in n\mathbb{Z}$ and $2a+b\in n\mathbb{Z}$. 
Then we have $3(a+b)\in n\mathbb{Z}$ and $a-b\in n\mathbb{Z}$ by adding and substituting them. 
If $n\notin 3\mathbb{Z}$, it follows from $3(a+b) \in n\mathbb{Z}$ that $a+b\in n\mathbb{Z}$, and we conclude $a, b\in n\mathbb{Z}$, which contradicts ${\rm gcd}(a,b,n)=1$. 
So we can assume $n\in 3\mathbb{Z}$. Put $n=3k$ with an integer $k\ge 2$. 
Since $a-b\in n\mathbb{Z}$, we may write $a=b+3kl$ with an integer $l$. 
Then $2a+b=3(b+2kl)$. Since $2a+b\in n\mathbb{Z}=3k\mathbb{Z}$, we have $b\in k\mathbb{Z}$. 
Then, it follows $a\in k\mathbb{Z}$, contradicting ${\rm gcd}(a,b,n)=1$.
\end{proof}

Let $X$ be a smooth projective surface and $\sigma\in \mathrm{Aut}(X)$ a holomorphic  automorphism of $X$ of order $n$. 
We denote by $\mathrm{Fix}(\sigma)$ the set of all fixed points of $\sigma$.

Take a point $x\in {\rm Fix}(\sigma)$ and an open neighborhood of $U$ of $x$ such that $\sigma(U)=U$. 
Let $(z_1,z_2)$ be a system of local coordinates on $U$ with $x=(0,0)$, and write $\sigma(z)=(\sigma_1(z_1,z_2), \sigma_2(z_1,z_2))$. If $\sigma_i(z_1,z_2)=a_{i,1}z_1+a_{i,2}z_2+\cdots$ is the expansion around $x$, then the Jacobian matrix at $x$ is given by
$$
(J\sigma)_x=\left(
\begin{array}{ccc}
a_{1,1} & a_{1,2} \\
a_{2,1} & a_{2,2} 
\end{array}
\right).
$$
Since $\sigma^n={\rm Id}$, we may assume that
$$
\left(
\begin{array}{ccc}
a_{1,1} & a_{1,2} \\
a_{2,1} & a_{2,2} 
\end{array}
\right)=\left(
\begin{array}{ccc}
\zeta^{k_1} & 0 \\
0 & \zeta^{k_2} 
\end{array}
\right),
$$
where $\zeta=\exp(2\pi\sqrt{-1}/n)$ and $k_1,k_2$ are integers satisfying $0\leq k_1\leq k_2\leq n-1$. Since $\sigma\neq \mathrm{Id}$, we have $k_2>0$. 
It is clear that $x$ is a smooth point on a $1$-dimensional fixed locus 
when $k_1=0$, and that it is an isolated fixed point when $k_1>0$. 
Since $\sigma$ is of order $n$, we have $\mathrm{gcd}(k_1,k_2,n)=1$. 
If the Jacobian matrix at $x$ has the canonical form described above, then we call $x$ a fixed point of type $(k_1,k_2)$.
Let $\rho_1\colon X_1\to X$ be the blow-up at a fixed point $x$ of type $(k_1,k_2)$. 
Put $E:=\rho_1^{-1}(x)$ and let $\sigma_1$ be the automorphism of $X_1$ of order $n$ induced by $\sigma$. 
Then one easily sees that $E\subset \mathrm{Fix}(\sigma_1)$ when $k_1=k_2$, and that there are exactly two isolated fixed points on 
$E$ of respective types $(k_1,k_2-k_1)$ and $(k_1-k_2,k_2)$ when 
$k_1\neq k_2$. 
With this remark, the following can be shown by using Lemma \ref{easylem} repeatedly.

\begin{lem} \label{kktypelem}
Let $X$ be a smooth projective surface with an automorphism $\sigma$ of order $n$. 
If there exists a birational morphism $\rho\colon \widetilde{X}\rightarrow X$ such that the automorphism $\widetilde{\sigma}$ on $\widetilde{X}$ induced by $\sigma$ has no isolated fixed points, then either $n\le 3$ or any isolated fixed point of $\sigma$ is of type $(k,k)$ for some $k$ $(0<k<n, \mathrm{gcd}(k,n)=1)$.
\end{lem}

We apply the above observation to the situation we are interested in. Before going further, we need some preparations.

Let $Y$ be a smooth projective surface and $R$ an effective divisor on $Y$ which is divisible by $n$ in the Picard group $\mathrm{Pic}(Y)$, that is, $R$ is linearly equivalent to $n\mathfrak{d}$ for some divisor $\mathfrak{d}\in \mathrm{Pic}(Y)$. 
Then we can construct a finite $n$-sheeted covering of $Y$ with branch locus $R$ as follows. Put
$\mathcal{A}=\bigoplus_{j=0}^{n-1}\mathcal{O}_Y(-j\mathfrak{d})$
and introduce a graded $\mathcal{O}_Y$-algebra structure on $\mathcal{A}$ by multiplying the section of $\mathcal{O}_Y(n\mathfrak{d})$ defining $R$. 
We call $Z:=\mathrm{Spec}_Y(\mathcal{A})$ equipped with the natural surjective morphism $\varphi\colon Z\to Y$ a  {\em classical 
$n$-cyclic covering} of $Y$ branched over $R$, according to \cite{bar}. 
Locally, $Z$ is defined by $z^n=r(x,y)$, where $r(x,y)$ denotes the 
local analytic equation of $R$. 
From this, one sees that $Z$ is normal if and only if $R$ is reduced, and $Z$ is smooth if and only if so is $R$. 
When $Z$ is smooth, we have
\begin{equation}\label{cyceq}
 \varphi^*R=nR_0,\; K_Z=\varphi^*K_Y+(n-1)R_0,\; \mathrm{Aut}(Z/Y)\simeq \mathbb{Z}/n\mathbb{Z}
\end{equation}
where $R_0$ is the effective divisor (usually called the ramification 
divisor) on $Z$ defined locally by $z=0$, and $\mathrm{Aut}(Z/Y)$ is the 
covering transformation group for $\varphi$.

\begin{defn}\normalfont \label{primdef}
A relatively minimal fibration $f\colon S\to B$ of genus $g\geq 2$ is called a 
primitive cyclic covering fibration of type $(g,h,n)$, if there exist a (not 
necessarily relatively minimal) fibration 
$\widetilde{\varphi}\colon \widetilde{W}\to B$ of genus $h\geq 0$, and a 
 classical $n$-cyclic covering 
$$
\widetilde{\theta}\colon \widetilde{S}=
\mathrm{Spec}_{\widetilde{W}}\left(\bigoplus_{j=0}^{n-1}
\mathcal{O}_{\widetilde{W}}(-j\widetilde{\mathfrak{d}})\right)\to \widetilde{W}
$$ 
branched over a smooth curve $\widetilde{R}\in |n\widetilde{\mathfrak{d}}|$
for some $n\geq 2$ and $\widetilde{\mathfrak{d}}\in 
\mathrm{Pic}(\widetilde{W})$ 
such that 
$f$ is the relatively minimal 
model of $\widetilde{f}:=\widetilde{\varphi}\circ \widetilde{\theta}$.
\end{defn}

\begin{rem}
Note that a finite cyclic covering between complex manifolds is not 
necessarily obtained in this way. For example, the most typical cyclic $n$-sheeted 
covering between complex lines given by $z\mapsto z^n$ branches over 
only two points ($0$ and $\infty$). This is why we put the 
 adjective ``primitive'' to our cyclic covering fibrations.
\end{rem}

Let $f\colon S\to B$ be a primitive cyclic covering fibration of type $(g,h,n)$. 
We freely use the notation in Definition \ref{primdef}. 
Let $\widetilde{F}$ and $\widetilde{\Gamma}$ be general fibers of 
$\widetilde{f}$ and $\widetilde{\varphi}$, respectively. 
Then the restriction map 
$\widetilde{\theta}|_{\widetilde{F}}\colon \widetilde{F}\to 
\widetilde{\Gamma}$ is a classical $n$-cyclic covering branched over 
$\widetilde{R}\cap \widetilde{\Gamma}$. 
Since the genera of $\widetilde{F}$ and $\widetilde{\Gamma}$ are $g$ and 
$h$, respectively, the Hurwitz formula gives us
\begin{equation}\label{r}
r:=\widetilde{R}\widetilde{\Gamma}=\frac{2(g-1-n(h-1))}{n-1}.
\end{equation}
Note that $r$ is a multiple of $n$. 
Let $\widetilde{\sigma}$ be a generator of $\mathrm{Aut}(\widetilde{S}/\widetilde{W})
\simeq \mathbb{Z}/n\mathbb{Z}$ and $\rho\colon \widetilde{S}\to S$ the natural 
birational morphism. 
By assumption, $\mathrm{Fix}(\widetilde{\sigma})$ is a disjoint 
union of smooth curves and $\widetilde{\theta}(\mathrm{Fix}(\widetilde{\sigma}))=\widetilde{R}$. Let $\varphi\colon W\to B$ be a relatively minimal model of 
$\widetilde{\varphi}$ and $\widetilde{\psi}\colon \widetilde{W}\to W$ the 
natural birational morphism. 
Since $\widetilde{\psi}$ is a succession of blow-ups, 
we can write $\widetilde{\psi}=\psi_1\circ \cdots \circ \psi_N$, where 
$\psi_i\colon W_i\to W_{i-1}$ denotes the blow-up at $x_i\in W_{i-1}$ 
$(i=1,\dots,N)$ with $W_0=W$ and $W_N=\widetilde{W}$. 
We define reduced curves $R_i$ on $W_i$ inductively as $R_{i-1}=(\psi_i)_*R_i$ starting 
from $R_N=\widetilde{R}$ down to $R_0=:R$. 
We also put $E_i=\psi_i^{-1}(x_i)$ and 
$m_i=\mathrm{mult}_{x_i}(R_{i-1})$ for $i=1,2,\dots, N$.

\begin{lem}\label{multlem}
With the above notation, the following hold for any $i=1,\dots, N$.

\smallskip

\noindent
$(1)$ Either $m_i\in n\mathbb{Z}$ or $m_i\in n\mathbb{Z}+1$.
Moreover, $m_i\in n\mathbb{Z}$ holds if and only if $E_i$ is not contained in $R_i$.

\smallskip

\noindent
$(2)$ 
 $R_i=\psi_i^*R_{i-1}-n\displaystyle{\left[\frac{m_i}{n}\right]}E_i$, 
 where $[t]$ is the greatest integer not exceeding $t$.

\smallskip

\noindent
$(3)$ There exists $\mathfrak{d}_i\in \mathrm{Pic}(W_i)$ such that 
$\mathfrak{d}_i=\psi_i^*\mathfrak{d}_{i-1}$ and $R_i\sim n\mathfrak{d}_i$, $\mathfrak{d}_N=\widetilde{\mathfrak{d}}$.
\end{lem}

\begin{proof}
Since $\widetilde{R}=R_N$ is reduced, every $R_i$ is reduced.
Set $\mathfrak{d}_N=\widetilde{\mathfrak{d}}$.
Since $\mathrm{Pic}(W_N)=\psi_N^*\mathrm{Pic}(W_{N-1})\bigoplus 
 \mathbb{Z}[E_N]$, there exist $\mathfrak{d}_{N-1}\in \mathrm{Pic}(W_{N-1})$ 
 and $d_N\in \mathbb{Z}$ such that $\mathfrak{d}_N=\psi_N^*\mathfrak{d}_{N-1}-d_NE_N$.
Then, inductively, we can take $\mathfrak{d}_{i-1}\in \mathrm{Pic}(W_{i-1})$ 
 and $d_i\in \mathbb{Z}$ such that 
 $\mathfrak{d}_i=\psi_i^*\mathfrak{d}_{i-1}-d_iE_i$, for $i=N, N-1,\dots, 1$.
Since $\widetilde{R}=R_N\sim 
 n\widetilde{\mathfrak{d}}=\psi_N^*(n\mathfrak{d}_{N-1})-nd_NE_N$ and 
 $R_{N-1}=(\psi_N)_*R_N$, we get $R_{N-1}\sim n\mathfrak{d}_{N-1}$, where the 
 symbol 
$\sim$ means linear equivalence of divisors.
Then, by induction, we have $R_i\sim n\mathfrak{d}_i$ for any $i$.

Now, if $E_i$ is not contained in $R_i$, then $R_i$ is the proper 
 transform of $R_{i-1}$ by $\psi_i$, and hence we have $m_i=nd_i$.
Similarly, if $E_i$ is contained in $R_i$, then, since $R_i$ is reduced, 
$R_i-E_i$ is the proper transform of $R_{i-1}$ by $\psi_i$, and hence 
we have $m_i=d_in+1$.
In either case, $d_i=[m_i/n]$.
\end{proof}

A smooth rational curve with self-intersection number $-k$ is called a $(-k)$-curve.
An irreducible curve on $\widetilde{S}$ is called $\widetilde{f}$-{\em vertical}
 if it is contained in a fiber of $\widetilde{f}$, 
otherwise, it is called $\widetilde{f}$-{\em horizontal}.

\begin{lem}\label{exclem}
Let $E$ be an $\widetilde{f}$-vertical $(-1)$-curve on $\widetilde{S}$ and put $L=\widetilde{\theta}(E)$.

\smallskip

\noindent
$(1)$ If $E\subset \mathrm{Fix}(\widetilde{\sigma})$, then $L$ is a 
 $\widetilde{\varphi}$-vertical $(-n)$-curve and 
 contained in $\widetilde{R}$.
Conversely, for any $\widetilde{\varphi}$-vertical $(-n)$-curve 
 $L'\subset \widetilde{R}$, there exists a $\widetilde{f}$-vertical $(-1)$-curve $E'$ such that 
$E'\subset \mathrm{Fix}(\widetilde{\sigma})$ and $\widetilde{\theta}^*L'=nE'$.

\smallskip

\noindent
$(2)$ If $E\not\subset \mathrm{Fix}(\widetilde{\sigma})$, then $L$ is a 
 $\widetilde{\varphi}$-vertical $(-1)$-curve and there exist
 $\widetilde{f}$-vertical $(-1)$-curves $E_2,\dots,E_n$ such that 
$\widetilde{\theta}^*L=E_1+E_2+\cdots+E_n$ and $E_1, E_2, \dots, E_n$ 
 are disjoint, where $E_1=E$.
Moreover, if $\overline{\rho}\colon \widetilde{S}\to \overline{S}$ and $\overline{\psi}\colon \widetilde{W}\to 
 \overline{W}$ denote the 
 contractions of $\cup_{i=1}^nE_i$ and $L$, respectively, then there 
 exists a natural classical $n$-cyclic covering 
 $\overline{\theta}\colon \overline{S}\to \overline{W}$ branched over 
 $\overline{R}=\overline{\psi}_*\widetilde{R}$ such that 
$\overline{\theta}\circ \overline{\rho}=\overline{\psi}\circ \widetilde{\theta}$.
\end{lem}

\begin{proof}
(1) Suppose that $E\subset \mathrm{Fix}(\widetilde{\sigma})$.
Then $L$ is $\widetilde{\varphi}$-vertical and $L\subset \widetilde{R}$. 
From the last, it follows  $\widetilde{\theta}^*L=nE$.
Since $nL^2=(\widetilde{\theta}^*L)^2=n^2E^2=-n^2$, we get $L^2=-n$.
Since $\widetilde{\theta}|_E\colon E\to L$ is an isomorphism, we have 
$L\simeq \mathbb{P}^1$ implying that $L$ is a $(-n)$-curve.
The rest of (1) may be clear.

(2) Suppose that $E\not\subset \mathrm{Fix}(\widetilde{\sigma})$.
Then either $\widetilde{\theta}^*L$ is irreducible or it consists of $n$ 
 irreducible curves.
The first alternative is impossible, because, if 
 $\widetilde{\theta}^*L=E$, we would have $L^2=-1/n$ which is absurd.
Therefore, $\widetilde{\theta}^*L$ consists of $n$ irreducible curves 
$E_1,\dots,E_n$. We may assume $E_1=E$.
Note that some power of $\widetilde{\sigma}$ maps $E$ isomorphically 
 onto $E_i$ for any $i$.
Therefore, every $E_i$ is an $\widetilde{f}$-vertical $(-1)$-curve, since so is $E$.
All the $E_i$'s must be contracted by $\rho$, since $g>0$.
It follows that $E_i$'s are mutually disjoint.
In particular, we see that $L\cap \widetilde{R}=\emptyset$.
Let $\overline{\rho}\colon \widetilde{S}\to \overline{S}$ and $\overline{\psi}\colon \widetilde{W}\to 
 \overline{W}$ be the 
 contractions of $\cup_{i=1}^nE_i$ and $L$, respectively.
Then $\overline{R}=\overline{\psi}_*\widetilde{R}$ is isomorphic to 
 $\widetilde{R}$. In particular, it is smooth and $\overline{\psi}^*\overline{R}=\widetilde{R}$.
There is a uniquely determined element $\overline{\mathfrak{d}}\in \mathrm{Pic}(\overline{W})$ 
satisfying $\widetilde{\mathfrak{d}}=\overline{\psi}^*\overline{\mathfrak{d}}$.
Then we have $\overline{R}\sim n\overline{\mathfrak{d}}$.
Hence we can construct a classical $n$-cyclic covering 
$\overline{\theta}\colon \mathrm{Spec}(\bigoplus_{j=0}^{n-1}\mathcal{O}_{\overline{W}}(-j\overline{\mathfrak{d}}))
\to \overline{W}$ branched over $\overline{R}$.
Clearly, it is isomorphic to the natural map $\overline{S}\to \overline{W}$.
\end{proof}

By this lemma, we can assume from the first time that all the 
$\tilde{f}$-vertical $(-1)$-curves are contained in $\mathrm{Fix}(\widetilde{\sigma})$.
In what follows, we tacitly assume it.
We decompose $\rho\colon \widetilde{S}\to S$ into a succession of blow-ups as 
$\rho=\rho_1\circ \cdots \circ \rho_k$, where $\rho_i\colon \overline{S}_i\to \overline{S}_{i-1}$ is 
a blow-up $(i=1,\dots, k)$, $\overline{S}_k=\widetilde{S}$ and $\overline{S}_0=S$.

\begin{lem}\label{autolem}
$\widetilde{\sigma}$ induces on $\overline{S}_i$ an automorphism $\sigma_i$ of 
 order $n$ for each $i$, and the center of $\rho_i$ is a fixed point of $\sigma_{i-1}$.
\end{lem}

\begin{proof}
Let $E_i$ be the exceptional $(-1)$-curve of $\rho_i$.
On $\overline{S}_k=\widetilde{S}$, we put $\sigma_k=\widetilde{\sigma}$. Then we 
 have $E_k\subset \mathrm{Fix}(\sigma_k)$ by assumption.
Then $\sigma_k$ induces an automorphism $\sigma_{k-1}$ of $\overline{S}_{k-1}$ of 
 order $n$.
If $\sigma_{k-1}(E_{k-1})\neq E_{k-1}$, then, as in the proof of 
 Lemma~\ref{exclem} (2), one can show 
 $E_{k-1}\cap \mathrm{Fix}(\sigma_{k-1})=\emptyset$.
This implies that $E_{k-1}$ can be regarded as a $(-1)$-curve on 
 $\widetilde{S}$ disjoint from $\mathrm{Fix}(\widetilde{\sigma})$, 
 contradicting our assumption.
Hence $\sigma_{k-1}(E_{k-1})=E_{k-1}$ and $\sigma_{k-1}$ 
 induces an automorphism $\sigma_{k-2}$ of $\overline{S}_{k-2}$ of order $n$.
Now, an easy inductive argument shows the assertion.
\end{proof}

We put $\sigma:=\sigma_0$.
Since $\tilde{\sigma}$ has no isolated fixed points, it follows from 
Lemma~\ref{kktypelem} applied to $\rho\colon \widetilde{S}\to S$ that either $n\leq 
3$ or every isolated fixed point of $\sigma$ is of type $(l,l)$ for some $l$.
When $n=2$, it is clear that every isolated fixed point is of type $(1,1)$.
For $n=3$, we can show the following:

\begin{lem}\label{n3lem}
If $n=3$, any isolated fixed point $x$ of $\sigma$ is of type either 
$(1,1)$ or $(2,2)$.
\end{lem}

\begin{proof}
Assume that $x$ is an isolated fixed point of type $(1,2)$.
Let $\rho_1\colon \overline{S}_1\to S$ be the blow-up at $x$.
Then the automorphism $\sigma_1$ induced by $\sigma$ has two isolated 
 fixed points of respective types $(1,1)$ and $(2,2)$ on $\rho_1^{-1}(x)$.
We need to blow $\overline{S}_1$ up at two such fixed points to get $\widetilde{S}$.
We denote the proper transform of $\rho_1^{-1}(x)$ by $E$, and $E_i$ the new 
 exceptional $(-1)$-curves coming from the fixed point of type $(i,i)$ $(i=1,2)$.
Then $E$ is a $(-3)$-curve not contained in 
 $\mathrm{Fix}(\widetilde{\sigma})$, while $E_1, E_2\subset \mathrm{Fix}(\widetilde{\sigma})$.
Let $L$, $L_1$, $L_2$ be the image of $E$, $E_1$, $E_2$ under 
 $\widetilde{\theta}$, respectively.
One sees easily that $L$ is a $(-1)$-curve not contained in 
 $\widetilde{R}$, and that $L_1, L_2$ are $(-3)$-curves contained in $\widetilde{R}$.
Furthermore, $L+L_1+L_2$ is contained in a fiber of $\widetilde{\varphi}$.

On the other hand, $\widetilde{\theta}|_E\colon E\to L$ is a classical 
 $3$-cyclic covering branched over $\widetilde{R}\cap L$.
Recall that $\widetilde{R}\cap L$ contains at least two points $L\cap L_1$ and 
 $L\cap L_2$.
Since $E$ and $L$ are smooth rational curves, we apply the Hurwitz 
 formula to see that it branches over exactly two points.
Therefore, $L$ meets $\widetilde{R}$ exactly at $L\cap (L_1\cup L_2)$. 
Since $L$ is a $\widetilde{\varphi}$-vertical $(-1)$-curve, we can assume that $\widetilde{\psi}$ contracts $L$.
By what we  have just seen, after the contraction, the image of $L$ will be a double point on the 
 branch locus.
This contradicts Lemma~\ref{multlem} (1).
\end{proof}
 
Therefore, we get the following lemma:

\begin{lem}\label{rholem}
Any irreducible curve contracted by $\rho$ is a $\widetilde{f}$-vertical 
 $(-1)$-curve contained in $\mathrm{Fix}(\widetilde{\sigma})$.
\end{lem}

\begin{proof}
We write $\rho=\rho_1\circ \cdots \circ \rho_k$ as before.
It suffices to show that the center of $\rho_i$ is an isolated fixed 
 point of $\sigma_{i-1}$ for any $i$.
Suppose that the center of $\rho_i$ is a smooth point of a 
 $1$-dimensional fixed locus of $\sigma_{i-1}$.
Let its type be $(0,l)$, $\mathrm{gcd}(l,n)=1$.
Then, on the exceptional $(-1)$-curve of $\rho_i$, we can find an 
isolated fixed point of type $(n-l,l)$.
This is impossible by Lemmas~\ref{kktypelem} and \ref{n3lem}.
\end{proof}

From this lemma, we can reconstruct $(\widetilde{S}, 
\widetilde{\sigma})$ from $(S,\sigma)$ by blowing up isolated 
fixed points of $\sigma$.

\section{Lower bound of the slope}

The purpose of this section is to show the following theorem:

\begin{thm}\label{boundthm}
Let $f\colon S\to B$ be a primitive cyclic covering fibration of type 
$(g, h, n)$.
If $h\geq 1$ and $g\geq (2n-1)(2hn+n-1)/(n+1)$, then 
$K_f^2\geq \lambda_{g,h,n}\chi_f$ holds, where
\begin{equation}\label{boundeq}
\lambda_{g,h,n}=\frac{24(n-1)(g-1)}{2(2n-1)(g-1)-n(n+1)(h-1)}.
\end{equation}
\end{thm}

\begin{rem}\normalfont
$(1)$ By (\ref{r}), the condition $g\geq (2n-1)(2hn+n-1)/(n+1)$ is equivalent to 
$r\geq 3g/(2n-1)+3$, and 
$$
\lambda_{g,h,n}=\frac{8}{1+r(n+1)/6(g-1)}.
$$
$(2)$ The case where $h=0$ needs a separate treatment. As we shall see
 later in Theorem~\ref{slopeeq}, the inequality
$K_f^2\geq \lambda_{g,h,n}\chi_f$ also holds for $h=0$.
\end{rem}

We keep the notation in the previous section and, for the time being, we 
do not exclude the case where $h=0$ from the consideration.
We obtain a classical $n$-cyclic covering $\theta_i\colon S_i\to W_i$ branched 
over $R_i$ by setting
$$
S_i=\mathrm{Spec}\left(\bigoplus_{j=0}^{n-1}\mathcal{O}_{W_i}(-j\mathfrak{d}_i)\right).
$$
Since $R_i$ is reduced, $S_i$ is a normal surface.

There exists a natural birational morphism $S_i\rightarrow S_{i-1}$. Set $S^\prime=S_0$, $\theta=\theta_0$, $\mathfrak{d}=\mathfrak{d}_0$ and $f^{\prime}=\varphi\circ\theta$.
$$\xymatrix{
\widetilde{S}=S_N \ar[d]^{\widetilde{\theta}} \ar[r] \ar@(ur,ul)[rrrr]^{\rho} & S_{N-1} \ar[d]^{\theta_{N-1}} \ar[r] & \cdots \ar[r] & S_0=S^{\prime} \ar[d]^{\theta} & S \ar[ddl]^f\\
\widetilde{W}=W_N \ar[r]^{\psi_N} \ar[drrr]^{\widetilde{\varphi}}                & W_{N-1} \ar[r]^{\psi_{N-1}}            & \cdots \ar[r]^{\psi_1} & W_0=W \ar[d]^{\varphi} \\
                                                                                                    &                                                &                              & B
}$$
From Lemma \ref{multlem}, we have
\begin{eqnarray}
K_{\widetilde{\varphi}}&=
&\widetilde{\psi}^{\ast}K_{\varphi}+\sum_{i=1}^N {\bf E}_i \label{kphi},\\
\widetilde{\mathfrak{d}}&=
&\widetilde{\psi}^{\ast}\mathfrak{d}-\sum_{i=1}^N \left[\frac{m_i}{n}\right]{\bf E}_i, \label{delta}
\end{eqnarray}
where ${\bf E}_i$ denotes the total transform of $E_i$.
Since
$$
K_{\widetilde{S}}=
\displaystyle{\widetilde{\theta}^{\ast}\left(K_{\widetilde{W}}+(n-1)\widetilde{\mathfrak{d}}\right)}
$$
and
$$
\chi(\mathcal{O}_{\widetilde{S}})=
n\chi(\mathcal{O}_{\widetilde{W}})+\frac{1}{2}\sum_{j=1}^{n-1}j\widetilde{\mathfrak{d}}(j\widetilde{\mathfrak{d}}+K_{\widetilde{W}}),
$$
we get
\begin{eqnarray}
K_{\widetilde{f}}^2&=
&n(K_{\widetilde{\varphi}}^2+2(n-1)K_{\widetilde{\varphi}}\widetilde{\mathfrak{d}}+(n-1)^2\widetilde{\mathfrak{d}}^2), \label{kftilde} \\
\chi_{\widetilde{f}}&=
&n\chi_{\widetilde{\varphi}}+\frac{1}{2}\sum_{j=1}^{n-1}j\widetilde{\mathfrak{d}}(j\widetilde{\mathfrak{d}}+K_{\widetilde{\varphi}}). \label{chiftilde}
\end{eqnarray}
Similarly, we have
\begin{eqnarray}
\omega_{f^{\prime}}^2&=
&n(K_{\varphi}^2+2(n-1)K_{\varphi}\mathfrak{d}+(n-1)^2\mathfrak{d}^2) \label{omegafprime} \\
\chi_{f^{\prime}}&=
&n\chi_{\varphi}+\frac{1}{2}\sum_{j=1}^{n-1}j\mathfrak{d}(j\mathfrak{d}+K_{\varphi}), \label{chifprime}
\end{eqnarray}
where $\omega_{f^{\prime}}$ denotes the relative dualizing sheaf of $f^{\prime}$.
Combining these equalities $(\ref{kphi})$,\dots,$(\ref{chifprime})$, we obtain
\begin{eqnarray}
\omega_{f^{\prime}}^2-K_{\widetilde{f}}^2&=
&n\sum_{i=1}^N \left((n-1)\left[\frac{m_i}{n}\right]-1\right)^2, \label{fprime-ftilde1}\\
\chi_{f^{\prime}}-\chi_{\widetilde{f}}&=
&\frac{1}{12}n(n-1)\sum_{i=1}^N \left[\frac{m_i}{n}\right]\left((2n-1)\left[\frac{m_i}{n}\right]-3\right) \label{fprime-ftilde2}.
\end{eqnarray}

\begin{lem} \label{fprimelem}
Assume that 
$$
g\ge \frac{2n-1}{n+1}(2hn+n-1)
$$
when $h>0$.
Then $\omega_{f^{\prime}}^2\ge \lambda_{g,h,n}\chi_{f^{\prime}}$,
where $\lambda_{g,h,n}$ is the rational number in $(\ref{boundeq})$.
\end{lem}

\begin{proof}
Suppose that $h=0$. Then $R$ is numerically equivalent to $(-r/2)K_{\varphi}+M\Gamma$ for some $M\in \frac{1}{2}\mathbb{Z}$
since $\varphi\colon W\rightarrow B$ is a $\mathbb{P}^1$-bundle and $K_W\Gamma=-2$, $R\Gamma=\widetilde{R}\widetilde{\Gamma}=r$. Moreover, $K_{\varphi}^2=0$ and $\chi_{\varphi}=0$. Hence we have
\begin{eqnarray*}
\omega_{f^{\prime}}^2&=&\frac{4(g-1)(n-1)}{n}M,\\
\chi_{f^{\prime}}&=&\frac{2(2n-1)(g-1)+n(n+1)}{6n}M
\end{eqnarray*}
from \eqref{r},\eqref{omegafprime} and \eqref{chifprime}. Thus, we get $\omega_{f^{\prime}}^2=\lambda_{g,0,n}\chi_{f^{\prime}}$. 

If $h=1$, then $K_{\varphi}^2=0$ and $\chi_{\varphi}=\chi(\mathcal{O}_W)$ since $\varphi$ is a relatively minimal elliptic surface. Then we have
\begin{eqnarray}
\omega_{f^{\prime}}^2-\lambda_{g,1,n}\chi_{f^{\prime}}&=&n\left(K_{\varphi}^2-\frac{12(n-1)}{2n-1}\chi_{\varphi}\right)+\frac{n(n-1)(n+1)}{2n-1}K_{\varphi}\mathfrak{d} \nonumber \\
&=&\frac{n(n-1)}{2n-1}((n+1)K_{\varphi}\mathfrak{d}-12\chi_{\varphi}). \label{h=1ineq1}
\end{eqnarray}
By the canonical bundle formula, $K_{\varphi}$ is numerically equivalent to $\chi(\mathcal{O}_W)\Gamma+\sum_{i=1}^l (1-1/k_i)\Gamma$, where 
$\{k_i\mid i=1,\ldots,l\}$ denotes the set of multiplicities of all multiple 
fibers of $\varphi$, $k_i\geq 2$. Hence, we have
$$K_{\varphi}\mathfrak{d}\ge \chi(\mathcal{O}_W)\Gamma\mathfrak{d}=\frac{2(g-1)}{n(n-1)}\chi(\mathcal{O}_W).$$
Thus, 
\begin{eqnarray}
&&\frac{n(n-1)}{2n-1}((n+1)K_{\varphi}\mathfrak{d}-12\chi_{\varphi}) \nonumber \\
&\ge& \frac{n(n-1)}{2n-1}\left(\frac{2(n+1)(g-1)}{n(n-1)}-12\right)\chi(\mathcal{O}_W) \nonumber \\
&=&\frac{2}{2n-1}((n+1)g-(2n-1)(3n-1))\chi(\mathcal{O}_W). \label{h=1ineq2}
\end{eqnarray}
From \eqref{h=1ineq1}, \eqref{h=1ineq2} and hypothesis $g\ge (2n-1)(3n-1)/(n+1)$, we have $\omega_{f^{\prime}}^2\ge \lambda_{g,1,n}\chi_{f^{\prime}}$.

Let $h\ge 2$. We compute $\omega_{f^{\prime}}^2-\lambda_{g,h,n}\chi_{f^{\prime}}$ by \eqref{omegafprime}, \eqref{chifprime} as follows:
\begin{eqnarray}
\omega_{f^{\prime}}^2-\lambda_{g,h,n}\chi_{f^{\prime}}&=&n(K_{\varphi}^2+2(n-1)K_{\varphi}\mathfrak{d}+(n-1)^2\mathfrak{d}^2) \nonumber \\
&&-\lambda_{g,h,n}\left(n\chi_{\varphi}+\frac{1}{4}n(n-1)K_{\varphi}\mathfrak{d}+\frac{1}{12}n(n-1)(2n-1)\mathfrak{d}^2\right) \nonumber \\
&=&n(K_{\varphi}^2-\lambda_{g,h,n}\chi_{\varphi})+\frac{n(n-1)}{4}(8-\lambda_{g,h,n})K_{\varphi}\mathfrak{d} \nonumber \\
&&+\frac{n(n-1)}{12}(12(n-1)-(2n-1)\lambda_{g,h,n})\mathfrak{d}^2. \label{comp}
\end{eqnarray}
Since the slope inequality of $\varphi$ gives us
$$K_{\varphi}^2\ge \frac{4(h-1)}{h}\chi_{\varphi},$$
we have
\begin{eqnarray}
K_{\varphi}^2-\lambda_{g,h,n}\chi_{\varphi}&=&\frac{h\lambda_{g,h,n}}{4(h-1)}\left(K_{\varphi}^2-\frac{4(h-1)}{h}\chi_{\varphi}\right)+\left(1-\frac{h\lambda_{g,h,n}}{4(h-1)}\right)K_{\varphi}^2 \nonumber \\
&\ge &\left(1-\frac{h\lambda_{g,h,n}}{4(h-1)}\right)K_{\varphi}^2. \label{slopeineq}
\end{eqnarray}
We consider the intersection matrix of $\{K_{\varphi},\mathfrak{d},\Gamma\}$
\begin{equation} \label{intmat}
\left(
\begin{array}{ccc}
  K_{\varphi}^2 & K_{\varphi}\mathfrak{d} & K_{\varphi}\Gamma \\
  K_{\varphi}\mathfrak{d} & \mathfrak{d}^2 & \mathfrak{d}\Gamma \\
  K_{\varphi}\Gamma & \mathfrak{d}\Gamma & 0
\end{array}
\right).
\end{equation}
By Arakerov's theorem, we have $K_{\varphi}^2\ge 0$ and then the matrix is not negative definite. Hence the determinant of \eqref{intmat} is non-negative, that is, we have
\begin{equation} \label{det}
2(K_{\varphi}\mathfrak{d})(\mathfrak{d}\Gamma)( K_{\varphi}\Gamma)-\mathfrak{d}^2(K_{\varphi}\Gamma)^2-(\mathfrak{d}\Gamma)^2K_{\varphi}^2\ge 0
\end{equation}
by the Hodge index theorem. Since
$$\mathfrak{d}\Gamma=\frac{r}{n}=\frac{2(g-1-n(h-1))}{n(n-1)}$$
and
$$K_{\varphi}\Gamma=2(h-1),$$
the inequality \eqref{det} is equivalent to
\begin{equation} \label{detequi}
2(g-1-n(h-1)) K_{\varphi}\mathfrak{d}-n(n-1)(h-1)\mathfrak{d}^2\ge \frac{1}{n(n-1)(h-1)}(g-1-n(h-1))^2K_{\varphi}^2.
\end{equation}
On the other hand, by the definition of $\lambda_{g,h,n}$, we have
\begin{eqnarray}
&&\frac{n(n-1)}{4}(8-\lambda_{g,h,n}) \nonumber \\
&=&\frac{n(n-1)(n+1)}{2(2n-1)(g-1)-n(n+1)(h-1)}2(g-1-n(h-1)) \label{comp1}
\end{eqnarray}
and
\begin{eqnarray}
&&\frac{n(n-1)}{12}(12(n-1)-(2n-1)\lambda_{g,h,n}) \nonumber \\
&=&\frac{n(n-1)(n+1)}{2(2n-1)(g-1)-n(n+1)(h-1)}(-n(n-1)(h-1)). \label{comp2}
\end{eqnarray}
Hence, combining \eqref{comp}, \eqref{slopeineq}, \eqref{detequi}, \eqref{comp1} and \eqref{comp2}, we get
\begin{eqnarray}
\omega_{f^{\prime}}^2-\lambda_{g,h,n}\chi_{f^{\prime}}&\ge &n\left(1-\frac{h\lambda_{g,h,n}}{4(h-1)}\right)K_{\varphi}^2 \nonumber \\
&&+\frac{(n+1)(g-1-n(h-1))^2}{(h-1)(2(2n-1)(g-1)-n(n+1)(h-1))}K_{\varphi}^2. \label{esti}
\end{eqnarray}
Since we have
\begin{equation*}
1-\frac{h\lambda_{g,h,n}}{4(h-1)}=\frac{2(g-1)(-hn+2h-2n+1)-n(n+1)(h-1)^2}{(h-1)(2(2n-1)(g-1)-n(n+1)(h-1))},
\end{equation*}
the right hand side of the inequality \eqref{esti} is
\begin{equation} \label{rhsesti}
\frac{(g-1)((n+1)g-(2hn+n-1)(2n-1))}{(h-1)(2(2n-1)(g-1)-n(n+1)(h-1))}K_{\varphi}^2.
\end{equation}
By the assumption $g\ge (2n-1)(2hn+n-1)/(n+1)$, \eqref{rhsesti} is non-negative.
\end{proof}

\begin{prfofthm} \normalfont
Let $\varepsilon$ be the number of blow-ups of $\rho\colon \widetilde{S}\rightarrow S$. Then we have
\begin{equation}
K_{\widetilde{f}}^2=K_f^2-\varepsilon.
\end{equation}
Using \eqref{fprime-ftilde1}, \eqref{fprime-ftilde2}, \eqref{comp1} and \eqref{comp2}, we can calculate $K_f^2-\lambda_{g,h,n}\chi_f$ as follows:
\begin{eqnarray}
K_f^2-\lambda_{g,h,n}\chi_f&=&K_{\widetilde{f}}^2-\lambda_{g,h,n}\chi_{\widetilde{f}}+\varepsilon \nonumber \\
&=&\omega_{f^{\prime}}^2-n\sum_{i=1}^N \left((n-1)\left[\frac{m_i}{n}\right]-1\right)^2-\lambda_{g,h,n}\chi_{f^{\prime}} \nonumber \\
&&+\lambda_{g,h,n}\frac{1}{12}n(n-1)\sum_{i=1}^N \left[\frac{m_i}{n}\right]\left((2n-1)\left[\frac{m_i}{n}\right]-3\right)+\varepsilon \nonumber \\
&=&\omega_{f^{\prime}}^2-\lambda_{g,h,n}\chi_{f^{\prime}}+\frac{1}{12}n(n-1)((2n-1)\lambda_{g,h,n}-12(n-1))\sum_{i=1}^N \left[\frac{m_i}{n}\right]^2 \nonumber \\
&&+\frac{1}{4}n(n-1)(8-\lambda_{g,h,n})\sum_{i=1}^N \left[\frac{m_i}{n}\right]-nN+\varepsilon \nonumber \\
&=&\omega_{f^{\prime}}^2-\lambda_{g,h,n}\chi_{f^{\prime}}+\frac{n^2(n-1)^2(n+1)(h-1)}{2(2n-1)(g-1)-n(n+1)(h-1)}\sum_{i=1}^N \left[\frac{m_i}{n}\right]^2 \nonumber \\
&&+\frac{2n(n-1)(n+1)(g-1-n(h-1))}{2(2n-1)(g-1)-n(n+1)(h-1)}\sum_{i=1}^N \left[\frac{m_i}{n}\right]-nN+\varepsilon. \label{comp3}
\end{eqnarray}
When $h\ge 1$, the right hand side of \eqref{comp3} increases monotonically with respect to the multiplicity $m_i$. So we may assume that $[m_i/n]=1$ for all $i=1,\ldots,N$. Then,
\begin{eqnarray*}
&&\frac{n^2(n-1)^2(n+1)(h-1)}{2(2n-1)(g-1)-n(n+1)(h-1)}\sum_{i=1}^N \left[\frac{m_i}{n}\right]^2\\
&&+\frac{2n(n-1)(n+1)(g-1-n(h-1))}{2(2n-1)(g-1)-n(n+1)(h-1)}\sum_{i=1}^N \left[\frac{m_i}{n}\right]-nN \\
&=&\frac{n^2(n-2)N}{2(2n-1)(g-1)-n(n+1)(h-1)}((n+1)(n-2)(h-1)+2(g-1))  \\
&\ge& 0. 
\end{eqnarray*}
Combining them with Lemma \ref{fprimelem}, we conclude the proof.
\qed \end{prfofthm}

\begin{rem}
From the proofs of Theorem~\ref{boundthm} and Lemma~\ref{fprimelem}, $K_f^2=\lambda_{g,h,n}\chi_f$ holds if and only if
the following $3$ conditions are satisfied:

\smallskip

\noindent
(1) Either $R$ is non-singular, i.e., $\widetilde{R}=R$ or $n=2$ and $R$ has only negligible singularities.

\smallskip

\noindent
(2) The equality sign of the slope inequality of $\varphi$ holds: $K_{\varphi}^2=\displaystyle{\frac{4(h-1)}{h}}\chi_{\varphi}$.

\smallskip

\noindent
(3) The intersection matrix of $\{K_{\varphi},\mathfrak{d},\Gamma\}$ is singular, i.e., 
$K_{\varphi},\mathfrak{d},\Gamma$ are linearly dependent in $\mathrm{NS}(W)_{\mathbb{R}}$.

\smallskip

\noindent
This also holds for $h=0$ as we shall see in Theorem~\ref{slopeeq}.
The above condition (2) means that $\varphi$ is locally trivial or hyperelliptic with the Horikawa index $0$ (cf.\ \cite{kon}, \cite{xiao2}).
We will give an example of primitive cyclic covering fibrations of type $(g,h,n)$ with the slope $\lambda_{g,h,n}$ and $\varphi$ being locally trivial in Example~\ref{sharpex}.
\end{rem}

\begin{exa} \label{sharpex} \normalfont
Let $B$ and $\Gamma$ be smooth curves of genus $b$ and $h$, respectively. Let $\mathfrak{d}_1$ and $\mathfrak{d}_2$ be divisors on $B$ and $\Gamma$ of degree $N$ and $M$, respectively. For $N$ and $M$ sufficiently large, the divisor $\mathfrak{d}:=p_1^\ast\mathfrak{d}_1+p_2^\ast\mathfrak{d}_2$ on $B\times\Gamma$ gives us a base point free linear system, where $p_1$ and $p_2$ are the natural projections from $B\times\Gamma$ to $B$ and to $\Gamma$, respectively. Thus, we can take a smooth divisor $R\in \left|n\mathfrak{d} \right|$ for any $n>0$ by Bertini's theorem. Hence we may construct a classical $n$-cyclic covering $\theta\colon S\rightarrow B\times\Gamma$ branched over $R$. Let $f\colon S\rightarrow B$ be the composite of $p_1$ and $\theta$. We will compute $K_f^2$ and $\chi_f$. Let $F$ be a general fiber of $f$ and $g$ the genus of $F$. Applying the Hurwitz formula to $\theta|_F\colon F\rightarrow \Gamma=\{t\}\times\Gamma$, we have
\begin{eqnarray*}
2g-2&=&n(2h-2)+(n-1)n(p_1^\ast\mathfrak{d}_1+p_2^\ast\mathfrak{d}_2)p_1^\ast t\\
&=&n(2h-2)+(n-1)nM.
\end{eqnarray*}
Hence
\begin{equation*}
M=\frac{2(g-1-n(h-1))}{n(n-1)}.
\end{equation*}
Since $K_{p_1}=p_2^\ast K_{\Gamma}$ and $\chi_{p_1}=0$, we obtain
\begin{eqnarray*}
K_f^2&=&(\theta^\ast(K_{p_1}+(n-1)\mathfrak{d}))^2  \\
&=&n(p_1^\ast(n-1)\mathfrak{d}_1+p_2^\ast((n-1)\mathfrak{d}_2+K_{\Gamma}))^2  \\
&=&2n(n-1)N((n-1)M+2(h-1)),
\end{eqnarray*}
and
\begin{eqnarray*}
\chi_f&=&n\chi_{p_1}+\frac{1}{2}\sum_{j=1}^{n-1}j\mathfrak{d}(j\mathfrak{d}+K_{p_1})  \\
&=&\frac{1}{4}n(n-1)\mathfrak{d} K_{p_1}+\frac{1}{12}n(n-1)(2n-1)\mathfrak{d}^2  \\
&=&\frac{1}{4}n(n-1)N(2h-2)+\frac{1}{12}n(n-1)(2n-1)2NM  \\
&=&\frac{1}{6}n(n-1)N(3(h-1)+(2n-1)M).
\end{eqnarray*}
Thus we get
\begin{eqnarray*}
\frac{K_f^2}{\chi_f}&=&\frac{12(2(h-1)+(n-1)M)}{3(h-1)+(2n-1)M}\\
&=&\frac{24(n-1)(g-1)}{2(2n-1)(g-1)-n(n+1)(h-1)}\\
&=&\lambda_{g,h,n}.
\end{eqnarray*}
This example implies that the bound of Theorem \ref{boundthm} on the slope is sharp when $h$ and $n$ are fixed arbitrarily.
\end{exa}

\section{Primitive cyclic covering fibrations of a ruled surface}

From now on, we concentrate on the case $h=0$ and consider primitive cyclic covering fibrations of type $(g,0,n)$.

In this section, we study singular points on the branch locus $R$ to know how $\widetilde{f}$-vertical $(-1)$-curves in the ramification divisor appear. As we have already seen, this amounts to observing how $\widetilde{\varphi}$-vertical 
$(-n)$-curves in $\widetilde{R}$ appear.

\bigskip

\begin{lem} \label{eltrlem}
There exists a relatively minimal model $\varphi\colon W\rightarrow B$ of $\widetilde{\varphi}$ such that
if $n=2$ and $g$ is even, then
$$
{\rm mult}_x(R)\le \frac{r}{2}=g+1
$$
for all $x\in R$, and otherwise,
$$
{\rm mult}_x(R_h)\le \frac{r}{2}=\frac{g}{n-1}+1
$$
for all $x\in R_h$, where $R_h$ denotes the horizontal part of $R$, that is, the sum of all $\varphi$-horizontal components of $R$.
\end{lem}

\begin{proof}
We take a relatively minimal model $\varphi\colon W\rightarrow B$ of $\widetilde{\varphi}$ arbitrarily. Suppose that there exists a point $x$ of $R$ at which $R$ has the multiplicity greater than $r/2$. Let $\Gamma$ be the fiber of $\varphi$ through $x$. We will perform the elementary transformation at $x$ as follows: Let $\psi_1\colon W_1\rightarrow W$ be the blow-up at $x$, $E_1=\psi_1^{-1}(x)$ the exceptional $(-1)$-curve and $E'_1$ the proper transform of $\Gamma$. Then, $E'_1$ is also a $(-1)$-curve. Let $\psi_1^\prime\colon W_1\rightarrow W'$ be 
the contraction of $E'_1$ and $\varphi' \colon W' \rightarrow B$ the induced fibration. Set $x^\prime=\psi_1^\prime(E_1^\prime)$ and $\Gamma^\prime=\psi_1^\prime(E_1)$. Then, $\Gamma^\prime$ is the fiber of $\varphi^\prime$ over $t=\varphi(\Gamma)$. Let $m$ denotes the multiplicity of $R$ at $x$. Put $R_1=\psi_1^\ast R-n[m/n]E_1$ and $R^\prime=(\psi_1^\prime)_\ast R_1$. Let $m^\prime$ denotes the multiplicity of $R^\prime$ at $x^\prime$. Then it follows from Lemma~\ref{multlem} that $R_1=(\psi_1^\prime)^\ast R^\prime-n[m^\prime/n]E_1^\prime$. Since $R\Gamma=R_1\psi_1^\ast \Gamma=R_1(E_1+E_1^\prime)$, we have
\begin{equation*}
\left[\frac{m}{n}\right]+\left[\frac{m^{\prime}}{n}\right]=\frac{r}{n}.
\end{equation*}
Moreover, since $m> r/2$, we get
\begin{equation*}
\left[\frac{m^\prime}{n}\right]\le \frac{r}{2n}\le \left[\frac{m}{n}\right].
\end{equation*}
Hence we have $m^\prime\le r/2+1$ from Lemma \ref{multlem}. 
Moreover, $m'\le r/2$ when $n=2$ and $g$ is even.
If $m^\prime=r/2+1$, then we have $m^\prime\in n\mathbb{Z}+1$ and hence $E_1^\prime$ is contained in $R_1$. 
In particular, $\Gamma$ is contained in $R$. 
Since $r/2\in n\mathbb{Z}$, we have $m=r/2+1$. 
Hence the multiplicity of $R_h$ at $x$ is $r/2$. 
If $m^\prime\le r/2$, we replace the relatively minimal model $W$ with $W^\prime$. 
Since the number of singularities of $R$ are finite, we obtain a relatively minimal model satisfying the condition inductively.
\end{proof}

In the sequel, we will tacitly assume that our relatively minimal model $\varphi\colon W\rightarrow B$ of $\widetilde{\varphi}$ enjoys the property of 
the lemma.

\begin{rem}\label{3.2}
Recall that $r$ is a multiple of $n$.
If $n\geq 3$ and $R$ has a singular point of multiplicity $m\geq 2$, then 
$m\leq r/2+1$ from Lemma~\ref{eltrlem}, while $m\in n\mathbb{Z}$ or $m\in n\mathbb{Z}+1$ by Lemma~\ref{multlem}. It follows that $r\geq 2n$ when $n\geq 3$ and $R$ has a singular point. If $r=n\geq 3$, then $R$ is non-singular and we have $K_f^2=\lambda_{g,0,n}\chi_f$ by Lemma~\ref{fprimelem}.
If $n=2$, then $r=2g+2\geq 6$ since $g\geq 2$.
\end{rem}

Recall that the gonality of a non-singular projective curve is the 
minimum of the degree of surjective morphisms to $\mathbb{P}^1$.
The gonality of a fibered surface is defined to be that of a general 
fiber (cf.\ \cite{cliff}).

\begin{lem}\label{gon}
The gonality of a primitive cyclic covering fibration of type $(g,0,n)$ is $n$, when $r\geq 2n$. 
\end{lem}

\begin{proof}
We use the so-called Castelnuovo-Severi inequality. 
Assume contrary that the general fiber $F$ 
 has a morphism onto $\mathbb{P}^1$ of degree $k<n$.
This together with the natural map $F\to F/\langle 
 \sigma|_F\rangle\simeq \mathbb{P}^1$ defines a morphism $\Phi:F\to \mathbb{P}^1\times \mathbb{P}^1$.  
If $\Phi$ is of degree $m$ onto the image $\Phi(F)$, then $m$ is a 
 common divisor of $n,k$ and $\Phi(F)$ is a 
 divisor of bidegree $(n/m,k/m)$. In particular, the arithmetic genus of 
 $\Phi(F)$ is given by $(n/m-1)(k/m-1)$.
Now, let $F'$ be the normalization of $\Phi(F)$.
Since the cyclic $n$-covering $F\to \mathbb{P}^1$ factors through $F'$, we 
 see that the induced covering $F'\to \mathbb{P}^1$ of degree $n/m$ is a 
 totally ramified covering branched over $r$ points.
Then, by the Hurwitz formula, we have $2g(F')=(n/m-1)(r-2)$. 
Since the genus $g(F')$ of $F'$ is not bigger than the arithmetic genus 
 of $\Phi(F)$, we get $r\leq 2(k/m)$, which is impossible, since $r\geq 2n$ and $k<n$.
A more careful study shows that the gonality pencil is unique when 
 $r\geq 3n$ and that the gonality of $F$ is not less than $n/2$ when $r=n$.
\end{proof}

As we saw in \S1, $\widetilde{f}$-vertical $(-1)$-curves in 
$\mathrm{Fix}(\widetilde{\sigma})$ are in 
one-to-one correspondence with $\widetilde{\varphi}$-vertical 
$(-n)$-curves in $\widetilde{R}$ via $\widetilde{\theta}$. 
So we need to know how $\widetilde{\varphi}$-vertical $(-n)$-curves in 
$\widetilde{R}$ appear during the process of ``modulo $n$ resolution'' 
$\widetilde{\psi}\colon \widetilde{W}\to W$ of $R$.

Let $L\simeq \mathbb{P}^1$ be a $\widetilde{\varphi}$-vertical curve contained in $\widetilde{R}$.
Then we have $\widetilde{\theta}^*L=nD$ for some $\widetilde{f}$-vertical $D\simeq \mathbb{P}^1$ 
contained in $\mathrm{Fix}(\widetilde{\sigma})$.
If $D^2=-a$, then $L^2=-an$.
The image of $L$ by the natural morphism $\widetilde{\psi}\colon \widetilde{W}\to W$ is either a point or a fiber of $\varphi$.
In either case, we define a curve $C$ and a number $c$ as follows.

\begin{itemize}
\item If $\widetilde{\psi}(L)$ is a point, then $L$ is the proper transform of an exceptional $(-1)$-curve, say, $E_j$.
Since $E_j^2=-1$ and $L^2=-an$, $E_j$ is blown up $an-1$ times in total to get $L$.
We put $C=E_j$ and $c=an-1$.

\item If $\widetilde{\psi}(L)$ is a fiber $\Gamma$ of $\varphi$, then $L$ is the proper transform of $\Gamma$.
Since $\Gamma^2=0$ and $L^2=-an$, $\Gamma$ is blown up $an$ times in total to get $L$.
We put $C=\Gamma$ and $c=an$.
\end{itemize}

\noindent
In the former case, since $E_j$ is contained in $R_j$, the multiplicity 
of $R_{j-1}$ at the point obtained by contracting 
$E_j$ is in $n\mathbb{Z}+1$.
We will drop $j$ and simply write $R$ instead of $R_j$ for the time being.

Let $x_1,\dots, x_l$ be all the singular points of $R$ on $C$, and $m_i$ the multiplicity of $R$ at $x_i$.
Clearly we have $1\leq l \leq c$.
We consider a local analytic branch $D$ of $R-C$ around $x_i$ whose multiplicity at $x_i$ is $m\geq 2$ (i.e., $D$ has a cusp at $x_i$).
Then we have one of the following:

\smallskip

\noindent
(i) $D$ is not tangent to $C$ at $x_i$. 
If we blow $x_i$ up, then the proper transform of $D$ does not meet that of $C$.
Hence we have $(DC)_{x_i}=m$, where $(DC)_{x_i}$ denotes the local intersection number of $D$ and $C$ at $x_i$.

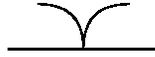
\setlength\unitlength{0.2cm}
\begin{figure}[H]
\begin{center}
 \begin{picture}(5,4)
 \put(-5,0){\line(1,0){10}}
 \qbezier(0,0)(0,3)(3,3)
 \qbezier(0,0)(0,3)(-3,3)
 \end{picture}
\end{center}
\end{figure}

\smallskip

\noindent
(ii) $D$ is tangent to $C$ at $x_i$. 
If we blow $x_i$ up, then one of the following three cases occurs.

\smallskip

(ii.1) The proper transform of $D$ is tangent to neither that of $C$ nor the exceptional $(-1)$-curve.

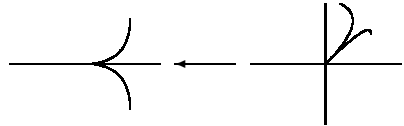
\setlength\unitlength{0.2cm}
\begin{figure}[H]
\begin{center}
 \begin{picture}(13,5)
 \put(-13,0){\line(1,0){10}}
 \qbezier(-8,0)(-5,0)(-5,3)
 \qbezier(-8,0)(-5,0)(-5,-3)
 \put(3,0){\line(1,0){10}}
 \qbezier(8,0)(11,3)(11,2)
 \qbezier(8,0)(11,3)(9,4)
 \put(8,-4){\line(0,1){8}}
 \put(2,0){\vector(-1,0){4}}
 \end{picture}
\end{center}
\end{figure}
\ \\

\smallskip

(ii.2) The proper transform of $D$ is tangent to the exceptional $(-1)$-curve.
Then the multiplicity $m'$ of the proper transform of $D$ at the singular point is less than $m$, and 
we have $(DC)_{x_i}=m+m'$.

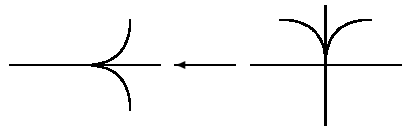
\setlength\unitlength{0.2cm}
\begin{figure}[H]
\begin{center}
 \begin{picture}(13,5)
 \put(-13,0){\line(1,0){10}}
 \qbezier(-8,0)(-5,0)(-5,3)
 \qbezier(-8,0)(-5,0)(-5,-3)
 \put(3,0){\line(1,0){10}}
 \qbezier(8,0)(8,3)(11,3)
 \qbezier(8,0)(8,3)(5,3)
 \put(8,-4){\line(0,1){8}}
 \put(2,0){\vector(-1,0){4}}
 \end{picture}
\end{center}
\end{figure}
\ \\

\smallskip

(ii.3) The proper transform of $D$ is still tangent to that of $C$.

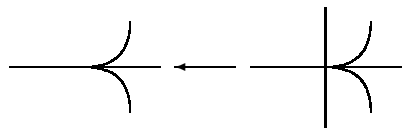
\setlength\unitlength{0.2cm}
\begin{figure}[H]
\begin{center}
 \begin{picture}(13,5)
 \put(-13,0){\line(1,0){10}}
 \qbezier(-8,0)(-5,0)(-5,3)
 \qbezier(-8,0)(-5,0)(-5,-3)
 \put(3,0){\line(1,0){10}}
 \qbezier(8,0)(11,0)(11,3)
 \qbezier(8,0)(11,0)(11,-3)
 \put(8,-4){\line(0,1){8}}
 \put(2,0){\vector(-1,0){4}}
 \end{picture}
\end{center}
\end{figure}
\ \\

\medskip

\noindent
We perform blowing-ups at $x_i$ and points infinitely near to it.
Then the case (ii.3) may occur repeatedly, but at most a finite number of times.
Suppose that the proper transform of $D$ becomes not tangent to that of $C$ just after $k$-th blow-up.
If the proper transform of $D$ is as in (ii.1) after $k$-th blow-up (or $D$ is as in (i) when $k=0$), then we have $(DC)_{x_i}=(k+1)m$.
If the proper transform of $D$ is as in (ii.2) after $k$-th blow-up, then we have $(DC)_{x_i}=km+m'$.
In either case, it is convenient to consider as if $D$ consists of $m$ local branches $D_1,\dots, D_m$ smooth at $x_i$ 
and such that $(D_jC)_{x_i}=k+1$ for $j=1,\dots, m$ in the former case and 
$$
(D_jC)_{x_i}=\left\{
\begin{array}{cl}
k, & \text{for }j=1,\dots, m-m',\\
k+1, & \text{for }j=m-m'+1,\dots, m
\end{array}
\right.
$$
in the latter case. 
We call $D_j$ a {\em virtual local branch} of $D$.

\begin{notation}\label{siknotation}
For a positive integer $k$, we let $s_{i,k}$ be the number of such virtual local branches $D_{\bullet}$ satisfying 
$(D_\bullet C)_{x_i}=k$, among those of all local analytic branches of $R-C$ around $x_i$.
Here, when $\mathrm{mult}_{x_i}(D)=1$, we regard $D$ itself as a virtual local branch.
We let $i_{\max}$ be the biggest integer $k$ satisfying $s_{i,k}\neq 0$.
\end{notation}

We put $x_{i,1}=x_i$ and $m_{i,1}=m_i$.
Let $\psi_{i,1}\colon W_{i,1}\to W$ be the blow-up at $x_{i,1}$ and put $E_{i,1}=\psi_{i,1}^{-1}(x_{i,1})$ and 
$R_{i,1}=\psi_{i,1}^*R-n[m_{i,1}/n]E_{i,1}$.
Inductively, we define $x_{i,j}$, $m_{i,j}$, $\psi_{i,j}\colon W_{i,j}\to W_{i,j-1}$, $E_{i,j}$ and $R_{i,j}$ 
to be the intersection point of the proper transform of $C$ and $E_{i,j-1}$, the multiplicity of $R_{i,j-1}$ at $x_{i,j}$, 
the blow-up of $W_{i,j-1}$ at $x_{i,j}$, the exceptional curve for $\psi_{i,j}$ and 
$R_{i,j}=\psi_{i,j}^*R_{i,j-1}-n[m_{i,j}/n]E_{i,j}$, respectively.
Put $i_{\mathrm{bm}}=\max\{j\mid m_{i,j}>1\}$, that is, the number of 
blowing-ups occuring over $x_i$.
We may assume that $i_{\mathrm{bm}}\geq (i+1)_{\mathrm{bm}}$ for 
$i=1,\dots,l-1$ after rearranging the index if necessary.

\begin{lem}\label{mslem}
With the above notation and assumption, the following hold.

\smallskip

\noindent
$(1)$ If $n\geq 3$, then $i_{\mathrm{bm}}=i_{\max}$ for all $i$. 
If $n=2$, then $i_{\mathrm{bm}}=i_{\max}$ $($resp.\ 
 $i_{\mathrm{bm}}=i_{\max}+1)$ if and only if $m_{i,i_{\max}}\in 
 2\mathbb{Z}$ $($resp.\ $m_{i,i_{\max}}\in 2\mathbb{Z}+1)$.

\smallskip

\noindent
$(2)$ $m_{i,1}=\sum_{k=1}^{i_{\max}}s_{i,k}+1$ and $m_{i,i_{\mathrm{bm}}}\in n\mathbb{Z}$.

\smallskip

\noindent
$(3)$ $m_{i,j}\in n\mathbb{Z}$ $($resp. $n\mathbb{Z}+1$$)$ if and only if $m_{i,j+1}=\sum_{k=j+1}^{i_{\max}}s_{i,k}+1$ 
$($resp. $\sum_{k=j+1}^{i_{\max}}s_{i,k}+2$$)$.

\smallskip

\noindent
$(4)$ $((R-C)C)_{x_i}=\sum_{k=1}^{i_{\mathrm{max}}}k s_{i,k}$.

\smallskip

\noindent
$(5)$ $c=\sum_{i=1}^l i_{\mathrm{bm}}$.
\end{lem}

\begin{proof}
(1) is clear from the definitions of $i_{\max}$ and $i_{\mathrm{bm}}$.
Since $m_{i,j}$ is the number of virtual local branches of $R_{i,j-1}$ through $x_{i,j}$, 
we get the first equality of (2) and (3) by Lemma~\ref{multlem}.
In (2), ``$+1$'' is the contribution of $C$.
If $m_{i,i_{\mathrm{max}}}\in n\mathbb{Z}+1$, then $x_{i,i_{\max+1}}$ is a double point, 
which is impossible when $n\geq 3$ by Lemma~\ref{multlem}.
Thus we have shown (2). (4) is clear from the definition of $s_{i,k}$.
For each $i$, we blow up $x_i\in C$ and its infinitely near points on 
 the proper transforms of $C$ exactly $i_{\max}$ times.
Hence (5).
\end{proof}

Put $t=(R-C)C$. It is the number of branch points $r$ 
if $\widetilde{\varphi}(L)$ is a fiber of $\varphi$, while it is the 
multiplicity at the point $x$ to which $C$ is contracted if 
$\widetilde{\varphi}(L)$ is a point.
By Lemma~\ref{mslem} (4), we get
$$
t=\sum_{i=1}^l\sum_{k=1}^{i_{\max}} ks_{i,k}.
$$
Let $c_i$ be the cardinality of $\{j\mid m_{i,j}\in n\mathbb{Z}+1\}$.
Clearly we have $0\leq c_i\leq i_{\mathrm{bm}}-1$.
Set $d_{i,j}=[m_{i,j}/n]$.

\begin{prop}\label{tcprop}
The following equalities hold:
$$
t+c+\sum_{i=1}^lc_i=\sum_{i=1}^l\sum_{j=1}^{i_{\mathrm{bm}}}m_{i,j},\quad
\frac{t+c}{n}=\sum_{i=1}^l\sum_{j=1}^{i_{\mathrm{bm}}}d_{i,j}.
$$
\end{prop}

\begin{proof}
It follows from Lemma~\ref{mslem} that
$$
t=\sum_{i=1}^l\sum_{k=1}^{i_{\max}} ks_{i,k} 
=\sum_{i=1}^l\left(\sum_{j=1}^{i_{\mathrm{bm}}}m_{i,j}-i_{\mathrm{bm}}-c_i\right) 
=\sum_{i=1}^l\left(\sum_{j=1}^{i_{\mathrm{bm}}}m_{i,j}-c_i\right)-c.
$$
Hence we have the first equality.
The second is clear from the first.
\end{proof}

\begin{lem}\label{mijlem}
The following hold:

\smallskip

\noindent
$(1)$ When $n\geq 3$, $m_{i,j}\geq m_{i,j+1}$ with equality sign holding if and only if 
 $s_{i,j}=0$ when $m_{i,j}\in n\mathbb{Z}$, or $s_{i,j}=1$ when 
 $m_{i,j}\in n\mathbb{Z}+1$.
When $n=2$, then $m_{i,j}+1\geq m_{i,j+1}$ with equality holding only if 
$m_{i,j}\in 2\mathbb{Z}+1$ and $m_{i,j+1}\in 2\mathbb{Z}$.

\smallskip

\noindent
$(2)$ If $m_{i,j-1}\in n\mathbb{Z}+1$ and $m_{i,j}\in n\mathbb{Z}$, then $m_{i,j}>m_{i,j+1}$.

\smallskip

\noindent
$(3)$ If $m_{i,j}=nd_{i,j}+1\in n\mathbb{Z}+1$, then 
 $d_{i,j}-d_{i,j+1}\geq n-3$.
\end{lem}

\begin{proof}
If $m_{i,j} < m_{i,j+1}$, then we have $s_{i,j}=0$ and $m_{i,j}+1=m_{i,j+1}$, since $m_{i,j}-m_{i,j+1}=s_{i,j}-1$, $s_{i,j}$, or $s_{i,j}+1$. 
Then $m_{i,j}\in n\mathbb{Z}+1$ and we get $m_{i,j+1}\in n\mathbb{Z}+2$ by Lemma~\ref{mslem} (2).
This contradicts Lemma~\ref{multlem} when $n\ge 3$. 
Hence $m_{i,j}\ge m_{i,j+1}$. 
The rest of (1) follows from Lemma~\ref{mslem} (2). 
If $m_{i,j-1}\in n\mathbb{Z}+1$ and $m_{i,j}\in n\mathbb{Z}$, then $m_{i,j}=\sum_{k=j}^{i_{\rm max}} s_{i,k}+2$ and $m_{i,j+1}=\sum_{k=j+1}^{i_{\rm max}} s_{i,k}+1$ by Lemma~\ref{mslem} (2). Then, $m_{i,j}-m_{i,j+1}=s_{i,j}+1 >0$ and hence (2) follows.
Suppose that $m_{i,j}=nd_{i,j}+1\in n\mathbb{Z}+1$. 
Let $C^\prime$ be the exceptional curve $E_{i,j}$ and define $x_{i,j}^\prime$, $m_{i,j}^\prime$, $d_{i,j}^\prime$, $c^\prime$ etc. on $C^\prime$ similarly to $C$. 
Since $C^\prime$ become a $(-a^\prime n)$-curve by blowing up for some $a^\prime\ge 1$, we have $c^\prime=a^\prime n-1$. 
We may assume $m_{i,j+1}=m_{1,1}^\prime$. Then we have
$$
\frac{m_{i,j}+c^\prime}{n}=\sum_{p=1}^{l^\prime} \sum_{q=1}^{p_{\rm max}^\prime} d_{p,q}^\prime\ge d_{i,j+1}+c^\prime-1.
$$
Hence we get
\begin{eqnarray*}
d_{i,j}-d_{i,j+1}&\ge& a^\prime(n-1)-2\\
                 &\ge& n-3.
\end{eqnarray*}
and thus (3) follows.
\end{proof}

\begin{prop}
$(1)$ If $g<(n-1)(an(n-1)/2-1)$, then there are no 
 $\widetilde{\varphi}$-vertical 
$(-an)$-curves in $\widetilde{R}$.

\smallskip

\noindent
$(2)$ If $(n-1)(an(n-1)/2-1)\leq g<(n-1)(an^2-(a+1)n-1)$,
then any $\widetilde{\varphi}$-vertical $(-an)$-curve in $\widetilde{R}$ 
 is the proper transform of a fiber of $\varphi$.

\smallskip

\noindent
$(3)$ Let $x$ be a singular point of multiplicity $m\in n\mathbb{Z}+1$.
If the proper transform $L$ of the exceptional curve obtained by blowing 
 up at $x$ is a $(-an)$-curve, then
$[m/n]\geq a(n-1)-1$.
\end{prop}

\begin{proof}
Let $L$ be a $\widetilde{\varphi}$-vertical $(-an)$-curve in $\widetilde{R}$.
Since we have 
$(t+c)/n=\sum_{i=1}^l\sum_{j=1}^{i_{\max}}d_{i,j}\geq c$ by Proposition~\ref{tcprop}, 
we get $t\geq (n-1)c$.

If $L$ is the proper transform of a fiber of $\varphi$, then $t=r$ and $c=an$. 
Hence $r\geq an(n-1)$. Then
$$
g\geq (n-1)\left(\frac{an(n-1)}{2}-1\right)
$$
by \eqref{r}.

If $L$ is the proper transform of an exceptional curve $E$ appeared in 
 $\widetilde{\psi}$, then $m\geq (an-1)(n-1)$, where $m$ denotes the 
 multiplicity of the branch locus at the point obtained by contracting $E$.
Hence (3) follows.
Moreover, we have $m\leq r/2+1$ by Lemmas~\ref{eltrlem} and \ref{mslem} (1).
Then we get
$$
g\geq (n-1)(an^2-(a+1)n-1)
$$
from the above two inequalities and \eqref{r}.

By an easy computation, one can show
$$
(n-1)(an^2-(a+1)n-1)\geq (n-1)\left(\frac{an(n-1)}{2}-1\right)
$$
with equality sign holding if and only if $(a,n)=(1,3)$.
In sum, we get (1) and (2).
\end{proof}

When $a=1$, we get the following:

\begin{cor}\label{gcor}
$(1)$ If $g<(n-2)(n-1)(n+1)/2$, then any irreducible components of 
 $\widetilde{R}$ is $\widetilde{\varphi}$-horizontal.

\smallskip

\noindent
$(2)$ If $(n-2)(n-1)(n+1)/2\leq g<(n-1)(n^2-2n-1)$, then any 
 $\widetilde{\varphi}$-vertical $(-n)$-curve in $\widetilde{R}$ is the 
 proper transform of a fiber of $\varphi$.

\smallskip

\noindent
$(3)$ The multiplicity of any singular point of the branch locus of type $n\mathbb{Z}+1$ is 
 greater than or equal to $(n-1)^2$.
\end{cor}

\section{Slope equality, Horikawa index and the local signature}

Let $f\colon S\to B$ be a primitive cyclic covering fibration of type $(g,0,n)$.
Firstly, we introduce singularity indices.

\begin{defn}[Singularity index $\alpha$]\label{sinddef}
Let $k$ be a positive integer.
For $p\in B$, we consider all the singular points (including infinitely near 
 ones) of $R$ on the fiber $\Gamma_p$ of $\varphi\colon W\to B$ over $p$.
We let $\alpha_k(F_p)$ be the number of singular points of multiplicity 
 either $kn$ or $kn+1$ among them, and call it the {\em $k$-th singularity 
 index} of $F_p$, the fiber of $f\colon S\to B$ over $p$.
Clearly, we have $\alpha_k(F_p)=0$ except for a finite number of $p\in B$.
We put $\alpha_k=\sum_{p\in B}\alpha_k(F_p)$ and call it the {\em $k$-th 
 singularity index} of $f$.

We also define $0$-th singularity index $\alpha_0(F_p)$ as follows.
Let $D_1$ be the sum of all $\widetilde{\varphi}$-vertical $(-n)$-curves 
 contained in $\widetilde{R}$ and put $\widetilde{R}_0=\widetilde{R}-D_1$.
Then, $\alpha_0(F_p)$ is the ramification index of 
 $\widetilde{\varphi}|_{\widetilde{R}_0}\colon \widetilde{R}_0\to B$ over $p$, 
 that is, the ramification index of 
 $\widetilde{\varphi}|_{(\widetilde{R}_0)_h}\colon (\widetilde{R}_0)_h\to B$ 
 over $p$ minus the sum of the topological Euler number of irreducible 
 components of $(\widetilde{R}_0)_v$ over $p$.
Then $\alpha_0(F_p)=0$ except for a finite number of $p\in B$, and we have
$$
\sum_{p\in B}\alpha_0(F_p)=(K_{\widetilde{\varphi}}+\widetilde{R}_0)\widetilde{R}_0
$$
by definition.
We put $\alpha_0=\sum_{p\in B}\alpha_0(F_p)$ and call it the {\em $0$-th singularity index} of $f$.
\end{defn}

\begin{rem}
The singularity indices defined above are somewhat different from those 
 defined in \cite{pi1} for $n=2$, because all singular points are ``essential'' if $n\geq 3$.
One can check that the value of $\alpha_k(F_p)$ does not depend on the 
 choice of the relatively minimal model of $\widetilde{W}\to B$ 
 satisfying Lemma~\ref{eltrlem} by the same argument for $n=2$ in \cite{pi1}.
\end{rem}

Let $\varepsilon(F_p)$ be the number of $(-1)$-curves contained in $F_p$, 
and put $\varepsilon=\sum_{p\in B}\varepsilon(F_p)$.
This is no more than the number of blowing-ups appearing in 
$\rho\colon \widetilde{S}\to S$.

Now, we compute the numerical invariants of $f$ using singularity indices.
Recall that $R$ is numerically equivalent to $-rK_\varphi/2+M\Gamma$ for 
some half-integer $M$.
We express $M$ in terms of singularity indices by calculating 
$(K_{\widetilde{\varphi}}+\widetilde{R})\widetilde{R}$ in two ways.
From \eqref{kphi} and \eqref{delta}, we have
\begin{align*}
(K_{\widetilde{\varphi}}+\widetilde{R})\widetilde{R}=&\;
\left(\widetilde{\psi}^*(K_\varphi+R)+\sum_{i=1}^N
\left(1-n\left[\frac{m_i}{n}\right]\right)\mathbf{E}_i\right)
\left(\widetilde{\psi}^*R-n\left[\frac{m_i}{n}\right]\mathbf{E}_i\right) \\
=&\;(K_\varphi+R)R-\sum_{i=1}^N n\left[\frac{m_i}{n}\right]
\left(n\left[\frac{m_i}{n}\right]-1\right) \\
=&\;\left(\left(1-\frac{r}{2}\right)K_\varphi+M\Gamma\right)
\left(-\frac{r}{2}K_\varphi+M\Gamma\right)
-\sum_{k\ge 1}nk(nk-1)\alpha_k.
\end{align*}
and, thus,
\begin{equation}\label{2cal1}
(K_{\widetilde{\varphi}}+\widetilde{R})\widetilde{R}=
2(r-1)M-n\sum_{k\ge 1}k(nk-1)\alpha_k.
\end{equation}
On the other hand, we have
\begin{equation}\label{2cal2}
(K_{\widetilde{\varphi}}+\widetilde{R})\widetilde{R}=(K_{\widetilde{\varphi}}+\widetilde{R}_0)\widetilde{R}_0
+D_1(K_{\widetilde{\varphi}}+D_1)=\alpha_0-2\varepsilon.
\end{equation}
Hence,
\begin{equation}\label{eq:4.4}
M=\frac{1}{2(r-1)}\left(\alpha_0+n\sum_{k\ge 1}k(nk-1)\alpha_k-2\varepsilon\right)
\end{equation}
by \eqref{2cal1} and \eqref{2cal2}.

Next, we compute $K_f^2$ and $\chi_f$.
We have
\begin{align*}
\widetilde{\mathfrak{d}}^2=&\;\mathfrak{d}^2-\sum_{i=1}^N\left[\frac{m_i}{n}\right]^2=
\frac{2r}{n^2}M-\sum_{k\ge 1}k^2\alpha_k,\\
\widetilde{\mathfrak{d}}K_{\widetilde{\varphi}}=&\;\mathfrak{d} K_\varphi+\sum_{i=1}^N\left[\frac{m_i}{n}\right]
=-\frac{2M}{n}+\sum_{k\ge 1}k\alpha_k,\\
K_{\widetilde{\varphi}}^2=&\;K_\varphi^2-N=-\sum_{k\ge 1}\alpha_k.
\end{align*}
Thus, we get
\begin{align*}
K_{\widetilde{f}}^2=&\; -n\sum_{k\ge 1}\alpha_k
+2n(n-1)\left(-\frac{2M}{n}+\sum_{k\ge 1}k\alpha_k\right)
+n(n-1)^2\left(\frac{2rM}{n^2}-\sum_{k\ge 1}k^2\alpha_k\right)\\
=&\;\frac{2(n-1)((n-1)r-2n)}{n}M-n\sum_{k\ge 1}((n-1)k-1)^2\alpha_k
\end{align*}
and
\begin{align*}
\chi_{\widetilde{f}}=&\;\frac{1}{12}n(n-1)(2n-1)
\left(\frac{2rM}{n^2}-\sum_{k\ge 1}k^2\alpha_k\right)
+\frac{1}{4}n(n-1)\left(-\frac{2M}{n}+\sum_{k\ge 1}k\alpha_k\right)\\
=&\;\frac{n-1}{6n}(r(2n-1)-3n)M-\frac{n(n-1)}{12}\sum_{k\ge 1}((2n-1)k^2-3k)\alpha_k
\end{align*}
by \eqref{kftilde} and \eqref{chiftilde}.
Hence, substituting \eqref{eq:4.4}, we obtain
\begin{align*}
K_f^2=&\;\frac{n-1}{r-1}\left(\frac{(n-1)r-2n}{n}(\alpha_0-2\varepsilon)
+(n+1)\sum_{k\ge 1}k(-nk+r)\alpha_k\right)
-n\sum_{k\ge 1}\alpha_k+\varepsilon,\\
\chi_f=&\;\frac{n-1}{12(r-1)}\left(\frac{(2n-1)r-3n}{n}(\alpha_0-2\varepsilon)+
(n+1)\sum_{k\ge 1}k(-nk+r)\alpha_k\right).
\end{align*}
These give us
\begin{equation}\label{ef}
e_f=12\chi_f-K_f^2=(n-1)\alpha_0+n\sum_{k\ge 1}\alpha_k-(2n-1)\varepsilon
\end{equation}
by Noether's formula.
Furthermore, we have
\begin{align*}
 &\; K_f^2-\lambda_{g,0,n}\chi_f \\
=&\;\frac{n}{(2n-1)r-3n}\sum_{k\ge 1}((n+1)(n-1)(-nk^2+rk)-(2n-1)r+3n)\alpha_k+\varepsilon,
\end{align*}
where $\lambda_{g,0,n}$ is the rational number defined in \eqref{boundeq}.
In this expression, the coefficient of $\alpha_k$ is non-negative, since 
$r\geq 2n$.
In fact, the quadratic function $(n+1)(n-1)(-nk^2+rk)-(2n-1)r+3n$ in $k$ 
is monotonically increasing in the interval $[1, r/2n]$, and its value at $k=1$ is 
$n(n-2)(r-n-2)\geq 0$.

We put
\begin{equation}\label{Horeq}
\mathrm{Ind}(F_p)=n\sum_{k\ge 1}\left(\frac{(n+1)(n-1)(r-nk)k}{(2n-1)r-3n}-1\right)\alpha_k(F_p)+
\varepsilon(F_p)
\end{equation}
Let $\mathcal{A}_{g,0,n}$ be the set of all fiber germs of primitive 
cyclic covering fibrations of type $(g,0,n)$.
Then \eqref{Horeq} defines a well-defined function $\mathrm{Ind}: 
\mathcal{A}_{g,0,n}\to \mathbb{Q}_{\geq 0}$ called the Horikawa index 
(cf.\  \cite{ak}).
Note that we have $\mathrm{Ind}(F_p)=0$ 
when either $r=n\geq 3$ (in this case $R$ is smooth) or $p\in B$ is general. We have shown the following:

\begin{thm}\label{slopeeq}
Let $f\colon S\to B$ be a primitive cyclic covering fibration of type $(g,0,n)$.
Then
$$
K_f^2=\lambda_{g,0,n}\chi_f+\sum_{p\in B}\mathrm{Ind}(F_p),
$$
where $\lambda_{g,0,n}$ is the rational number in $\eqref{boundeq}$ 
 and $\mathrm{Ind}\colon \mathcal{A}_{g,0,n}\to \mathbb{Q}_{\geq 0}$ denotes the 
 Horikawa index defined by $\eqref{Horeq}$.
\end{thm}

\begin{rem}
As we saw in Lemma~\ref{gon}, the gonality of $f$ is $n$ when $r\geq 2n$.
Therefore, the lower bound of the slope of $n$-gonal fibrations cannot 
 exceed 
$$
\lambda_{g,0,n}=\frac{12(n-1)}{2n-1}\left(1-\frac{n(n+1)}{2(2n-1)(g-1)+n(n+1)}\right).
$$
When $n=2$, we have $\lambda_{g,0,2}=4-4/g$ and, therefore, the above theorem recovers the slope 
 equality for hyperelliptic fibrations.
When $n=3$, we have $\lambda_{g,0,3}=24(g-1)/(5g+1)$ which coincides with the lower bound of the slope of semi-stable trigonal 
 fibrations in \cite{stankova}.
This is expected to hold also for unstable 
 ones (cf.\ \cite[p. 20]{trigonal}) and ours serves a new evidence for that.
When $n=4$, we have $\lambda_{g,0,4}=36(g-1)/(7g+3)$ which is strictly 
 greater than the bound $24(g-1)/(5g+3)$ given in \cite{beorchia-zucconi} for semi-stable 
 (non-factorized) tetragonal fibrations.
\end{rem}

Now, we state a topological application of the slope equality.
For an oriented compact real $4$-dimensional manifold $X$, 
the signature ${\rm Sign}(X)$ is defined to be the number of positive eigenvalues minus 
the number of negative eigenvalues of the intersection form on $H^2(X)$. 
Using the singularity indices, we observe the local concentration of ${\rm Sign}(S)$ on a finite number of fiber germs.

\begin{cor} \label{signcor}
Let $f\colon S\to B$ be a primitive cyclic covering fibration of type $(g,0,n)$. Then,
\begin{equation*}
{\rm Sign}(S)=\sum_{p\in B}\sigma(F_p),
\end{equation*}
where $\sigma\colon \mathcal{A}_{g,0,n}\rightarrow \mathbb{Q}$ is defined by
\begin{eqnarray*}
\sigma(F_p)&=&\frac{-(n-1)(n+1)r}{3n(r-1)}\alpha_0(F_p)+\sum_{k\ge 1}\left(\frac{(n-1)(n+1)(-nk^2+rk)}{3(r-1)}-n\right)\alpha_k(F_p)\\
&&+\frac{1}{3n(r-1)}((n+2)(2n-1)r-3n)\varepsilon(F_p),
\end{eqnarray*}
which is called the local signature of $F_p$.
\end{cor}

\begin{proof}
Once we have the Horikawa index, we can define the local signature 
 according to \cite{ak} as follows.
By the index theorem (cf.\ \cite[p.~126]{pag}), we have 
$$
\mathrm{Sign}(S)=\sum_{p+q\equiv 0 (\mathrm{mod} 2)}h^{p,q}(S)=K_f^2-8\chi_f.
$$
Put $\lambda=\lambda_{g,0,n}$. Then $\lambda<12$.
From the slope equality $K_f^2=\lambda \chi_f+\mathrm{Ind}$, where $\mathrm{Ind}=\sum_{p\in B}\mathrm{Ind}(F_p)$, and 
 Noether's formula $12\chi_f=K_f^2+e_f$, we get
$$
K_f^2=\frac{12}{12-\lambda}\mathrm{Ind}+\frac{\lambda}{12-\lambda}e_f,\quad
\chi_f=\frac{1}{12-\lambda}\mathrm{Ind}+\frac{1}{12-\lambda}e_f.
$$
Then
$$
\mathrm{Sign}(S)=\frac{4}{12-\lambda}\mathrm{Ind}-\frac{8-\lambda}{12-\lambda}e_f.
$$
Substituting \eqref{ef} and \eqref{Horeq}, the desired equality follows by 
 the definition of $\sigma(F_p)$.
\end{proof}

\section{Upper bound of the slope}

Let $f\colon S \to B$ be a primitive cyclic covering fibration of type $(g,0,n)$.
In this section, we give an upper bound of the slope of $f$.
Let the situation be as in the previous section.

For a vertical divisor $T$ and $p\in B$, we denote by $T(p)$ the 
greatest subdivisor of $T$ consisting of components of the fiber over $p$.
Then $T=\sum_{p\in B}T(p)$.
We consider a family $\{L^{i}\}_i$ of vertical irreducible curves in 
$\widetilde{R}$ over $p$ satisfying:

\smallskip

(i) $L^{1}$ is the proper transform of the fiber $\Gamma_p$ or a 
$(-1)$-curve $E^{1}$.

\smallskip

(ii) For $i\geq 2$, $L^{i}$ is the proper transform of an exceptional 
$(-1)$-curve $E^{i}$ that contracts to a point $x^{i}$ on $C^{k}$ 
or its proper transform for some $k<i$, where we let $C^{1}$ to be 
$E^{1}$ or $\Gamma_p$ according to whether $L^{1}$ is the proper 
transform of which curve, and $C^{j}=E^{j}$ for $j<i$.

\smallskip

(iii) $\{L^{i}\}_i$ is the largest among those satisfying (i) and (ii).

\medskip

The set of all vertical irreducible curves in $\widetilde{R}$ over $p$ 
is decomposed into the disjoint union of such families uniquely.
We denote it as 
$$
\widetilde{R}_v(p)=D^1(p)+\cdots+D^{\eta_p}(p),\quad D^t(p)=\sum_{k\ge 1}L^{t,k},
$$
where $\eta_p$ denotes the number of the decomposition and $\{L^{t,i}\}_i$ satisfies (i), (ii), (iii).
Let $C^{t,k}$ be the exceptional $(-1)$-curve or the fiber $\Gamma_p$ the proper transform of which is $L^{t,k}$.

\begin{defn}
Let $j^t_{a}(F_p)$ be the number of irreducible curves with self-intersection number $-an$ contained in $D^t(p)$.
Put $j_{a}(F_p):=\sum_{t=1}^{\eta_p}j^{t}_{a}(F_p)$, $j^{t}(F_p):=\sum_{a\ge 1}j^{t}_{a}(F_p)$ and $j(F_p):=\sum_{t=1}^{\eta_p}j^{t}(F_p)$..

Let $\alpha_0^{+}(F_p)$ be the ramification index of 
$\widetilde{\varphi}:\widetilde{R}_h\to B$ over $p$.
It is clear that $\varepsilon(F_p)=j_{1}(F_p)$ and $\alpha(F_p)=\alpha_0^{+}(F_p)
-2\sum_{a\ge 2}j_{a}(F_p)$.

Let $\iota^t(F_p)$ and $\kappa^t(F_p)$ be the number of singular points 
 over $p$ of types $n\mathbb{Z}$ and $n\mathbb{Z}+1$, respectively, at 
 which two proper transforms of $C^{t,k}$'s meet. 
We put $\iota(F_p)=\sum_{t=1}^{\eta_p}\iota^t(F_p)$ and $\kappa(F_p)=\sum_{t=1}^{\eta_p}\kappa^t(F_p)$.
\end{defn}

\begin{lem} \label{jplem}
The following hold:

\smallskip

\noindent
$(1)$ $\iota(F_p)=j(F_p)-\eta_p$.

\smallskip

\noindent
$(2)$ $\alpha^{+}_0(F_p)\ge (n-2)(j(F_p)-\eta_p+2\kappa(F_p))$.

\smallskip

\noindent
$(3)$  $\sum_{k\ge 1}\alpha_k(F_p)\ge \sum_{a\ge 1}(an-2)j_{a}(F_p)+2\eta_p-\kappa(F_p)$.
\end{lem}

\begin{proof}
We consider the following graph $\mathbf{G}^t$:
The vertex set $V(\mathbf{G}^t)$ is defined by the symbol set $\{v^{t,k}\}_{k=1}^{j^t(F_p)}$.
The edge set $E(\mathbf{G}^t)$ is defined by the symbol set $\{e_x\}_{x}$, where $x$ moves among all singularities contributing to $\iota^{t}(F_p)$. 
If $C^{t,k}$ or a proper transform of it meets that of $C^{t,k'}$ at a singularity $x$ of type $n\mathbb{Z}$, the edge $e_x$ connects $v^{t,k}$ and $v^{t,k'}$.
By the definition of the decomposition $\widetilde{R}_v(p)=D^1(p)+\cdots+D^{\eta_p}(p)$, the graph $\mathbf{G}^t$ is connected for any $t=1,\dots, \eta_p$.
Since this graph $\mathbf{G}^t$ has no loops, we have $\iota^{t}(F_p)=j^{t}(F_p)-1$.
Thus, we get (1) by summing it up for $t$.

When $L^{t,k}$ is a $(-an)$-curve, $L^{t,k}$ is obtained by blowing 
 $C^{t,k}$ up 
$an-1$ (or $an$ when $k=1$ and 
 $C^{t,1}=\Gamma_p$) times.
Thus, disregarding overlaps, $D^t(p)$ is obtained by blowing up 
$$
\sum_{a\ge 1}(an-1)j^t_a(F_p) \ (\text{or }\sum_{a\ge 1}(an-1)j^t_a(F_p)+1)
$$
times.
Thus, the number of singular points needed to obtain $D^t(p)$ can 
 be expressed as 
$$
\sum_{a\ge 1}(an-1) j^t_a(F_p)+1-\iota^t(F_p)-\kappa^t(F_p)
=\sum_{a\ge 1}(an-2) j^t_a(F_p)+2-\kappa^t(F_p).
$$
Then, the number of singular points needed to obtain 
 $\widetilde{R}_v(p)$ is
$$
\sum_{t=1}^{\eta_p}\left(\sum_{a\ge 1}(an-2)j^t_a(F_p)+2-\kappa^t(F_p)\right)
=\sum_{a\ge 1}(an-2)j_a(F_p)+2\eta_p-\kappa(F_p).
$$
This give us (3).

It remains to show (2). 
Let $\widetilde{\Gamma}_p=\sum_{i} m_iG_i$ be the irreducible decomposition. Then we have
\begin{align*}
\alpha_0^{+}(F_p)=&\;r-\#(\mathrm{Supp}(\widetilde{R}_h)\cap \mathrm{Supp}(\widetilde{\Gamma}_p)) \\
=&\; \sum_{i} m_i\widetilde{R}_hG_i-\#(\mathrm{Supp}(\widetilde{R}_h)\cap \mathrm{Supp}(\cup_i G_i)) \\
\ge& \sum_{i} (m_i-1)\widetilde{R}_hG_i.
\end{align*}
Let $x^{t,1},\dots, x^{t,\iota^t(F_p)}$ be all singular points over $p$ of type $n\mathbb{Z}$ at 
 which two proper transforms of $C^{t,k}$'s meet. Let $E^{t,k}$ be the exceptional curve obtained by blowing up at $x^{t,k}$ and $m^{t,k}$ the multiplicity of the fiber over $p$ along $E^{t,k}$. 
Let us estimate $\sum_{t=1}^{\eta_p}\sum_{k=1}^{\iota^t(F_p)} (m^{t,k}-1)\widetilde{R}_h\widehat{E}^{t,k}.$ 
If there exists a singular point of type $n\mathbb{Z}$ on $E^{t,k}$, we replace $E^{t,k}$ to the exceptional curve $E$ obtained by blowing up at this point. Repeating this procedure, we may assume that there exist no singular points of type $n\mathbb{Z}$ on $E^{t,k}$. If there exists  a singular point of type $n\mathbb{Z}+1$ on 
 $E^{t,k}$, the proper transform of the exceptional $(-1)$-curve obtained by blowing 
 up at this point belongs to 
 other $D^u(p)$ and becomes $L^{u,1}$ in $D^u(p)$.
Since the multiplicity of $\widetilde{\Gamma}_p$ along it is not less 
 than $m^{t,k}>1$, we do not have to consider this situation. 
Thus, we may assume that there exist no singularities on $E^{t,k}$ and we have $\widetilde{R}_h\widehat{E}^{t,k}\ge n-2$.
On the other hand, one sees that $\sum_{k=1}^{\iota^t(F_p)}m^{t,k}\ge 2\iota^t(F_p)+2\kappa^t(F_p)$. Hence, we have
\begin{align*}
\sum_{i} (m_i-1)\widetilde{R}_hG_i \ge& \sum_{t=1}^{\eta_p}\sum_{k=1}^{\iota^t(F_p)} (m^{t,k}-1)\widetilde{R}_h\widehat{E}^{t,k}      \\
 \ge& (n-2)\sum_{t=1}^{\eta_p}\sum_{k=1}^{\iota^t(F_p)} (m^{t,k}-1)  \\
 \ge& (n-2)\sum_{t=1}^{\eta_p} (\iota^t(F_p)+2\kappa^t(F_p))  \\
 =&\; (n-2)(\iota(F_p)+2\kappa(F_p)).
\end{align*}
Since $\iota(F_p)=j(F_p)-\eta_p$, we get (2).
\end{proof}

Using Lemma~\ref{jplem}, we give an upper bound of the slope:

\begin{thm}\label{upperthm}
Let $f\colon S\to B$ be a primitive cyclic covering fibration of type $(g,0,n)$ and assume $n\ge 4$.
Put $r:=\displaystyle{\frac{2g}{n-1}+2}$,
$\delta:=\left\{\begin{array}{l}
0, \ \text{if}\ r\in 2n\mathbb{Z}, \\
1, \ \text{if}\ r\not\in 2n\mathbb{Z}. \\
\end{array}
\right.$
Then, the following hold:

\smallskip

\noindent
$(1)$ If $n\le r< n(n-1)$, then
$$
K_f^2\leq \left(12-\frac{48n^2(r-1)}{(n-1)(n+1)(r^2-\delta n^2)}\right)\chi_f.
$$

\smallskip

\noindent
$(2)$ If $r\ge n(n-1)$, then
$$
K_f^2\leq \left(12-\frac{48n(n-1)(r-1)}{n(n+1)r^2-8(2n-1)r+24n-\delta n^3(n+1)}\right)\chi_f.
$$
\end{thm}

\begin{proof}
First, assume that $r\ge n(n-1)$. We put 
$$
\mu=\frac{48n(n-1)(r-1)}{(n(n+1)r^2-8(2n-1)r+24n-\delta n^3(n+1)},\quad
\mu'=\frac{n-1}{12(r-1)}\mu.
$$
By the formulae for $\chi_f$ and $e_f$ obtained in the previous 
 section, we get
\begin{align*}
&\;(12-\mu)\chi_f-K_f^2=e_f-\mu\chi_f \\
=&\;(n-1)\alpha_0+n\sum_{k\ge 1}\alpha_k-(2n-1)
 \varepsilon \\
&\;-\mu'\left(\frac{r(2n-1)-3n}{n}\alpha_0+(n+1)\sum_{k\ge 1}(-nk^2+rk)\alpha_k
-\frac{2(r(2n-1)-3n)}{n}\varepsilon\right) \\
=&\;\left(n-1-\frac{r(2n-1)-3n}{n}\mu'\right)\alpha_0 
+\sum_{k\ge 1}\left(Q(k)+n-
\frac{(n+1)(r^2-\delta n^2)}{4n}\mu'\right)\alpha_k \\
&\;-\left(2n-1-\frac{2(r(2n-1)-3n)}{n}\mu'\right)\varepsilon, \\
\end{align*}
where
$$
Q(k)=\mu'\left(n(n+1)\left(k-\frac{r}{2n}\right)^2-\frac{n(n+1)\delta}{4}\right)\ge 0.
$$
Therefore,
\begin{equation}\label{eq:5.2}
(12-\mu)\chi_f-K_f^2\geq A_n\alpha_0+B_n\sum_{k\ge 1}\alpha_k-(2A_n+1)\varepsilon,
\end{equation}
where 
$$
A_n=n-1-\frac{r(2n-1)-3n}{n}\mu',\quad B_n=n-\frac{(n+1)(r^2-\delta n^2)}{4n}\mu'.
$$
By the definitions of $\mu$ and $\mu'$, we see that $A_n$ and $B_n$ are positive 
 rational numbers satisfying 
\begin{equation}\label{eq:5.3}
-2A_n+nB_n-1=0.
\end{equation}
We shall show that the right hand side of \eqref{eq:5.2} is non-negative 
 by estimating it on each fiber $F_p$.
By Lemma~\ref{jplem}, we have
\begin{align*}
& A_n\alpha_0(F_p)+B_n\sum_{k\ge 1}\alpha_k(F_p)-(2A_n+1)\varepsilon(F_p)\\
\geq &\;\sum_{a\ge 2}((n-4)A_n+(an-2)B_n)j_a(F_p)+((n-4)A_n+(n-2)B_n-1)j_1(F_p)\\
&\; -((n-2)A_n-2B_n)\eta_p+(2(n-2)A_n-B_n)\kappa(F_p)\\
=&\;\sum_{a\ge 2}(-2A_n+anB_n)j_a(F_p)+(-2A_n+nB_n-1)j_1(F_p)\\
&\; +((n-2)A_n-2B_n)(j(F_p)-\eta_p)+(2(n-2)A_n-B_n)\kappa(F_p).\\
\end{align*}
When $n\geq 4$, one can show $(n-2)A_n-2B_n>0$ and thus all coefficients of $j_a(F_p)$, $j(F_p)-\eta_p$, $\kappa(F_p)$ are non-negative by \eqref{eq:5.3}. Hence we get (1).

Assume that $n\le r<n(n-1)$. We put 
$$
\mu=\frac{48n^2(r-1)}{(n-1)(n+1)(r^2-\delta n^2)},\quad
\mu'=\frac{n-1}{12(r-1)}\mu
$$
and 
$$
A_n=n-1-\frac{r(2n-1)-3n}{n}\mu',\quad B_n=n-\frac{(n+1)(r^2-\delta n^2)}{4n}\mu'.
$$
Clearly, $A_n>0$ and $B_n=0$. By Corollary \ref{gcor}, we get $j(F_p)=0$ for any $p\in B$. Thus we get
\begin{equation*}
(12-\mu)\chi_f-K_f^2\geq A_n\alpha_0^{+}+B_n\sum_{k\ge 1}\alpha_k
=A_n\alpha_0^{+}\ge 0,
\end{equation*}
which is the desired inequality.
\end{proof}

\section{Appendix}

Let $f:S\to B$ be a fibration of genus $g\geq 2$.
Put $e_f(F_p)=e(F_p)-e(F)$ for any fiber $F_p$, where $F$ is a general fiber.
It is well-known that $e_f(F_p)\geq 0$ with the equality holding if and 
only if $F_p$ is a smooth curve of genus $g$, and $e_f=\sum_{p\in B}e_f(F_p)$.
In \S4, we have defined the Horikawa index and the local signature for primitive cyclic covering fibrations of type $(g,0,n)$ using singularity indices.
However, in general, it is not known whether Horikawa indices or local signatures, if they exist, are unique (cf.\ \cite{ak}).
As to primitive cyclic covering fibrations of type $(g,0,n)$, we have obtained another local concentration of $e_f$ in \eqref{ef}.
Therefore, we may have two apparently distinct expressions of $\mathrm{Sign}(S)$ in this case from the proof of Corollary~\ref{signcor}.
In this appendix, we shall show that two expressions coincide. Namely,

\begin{prop}
Let $f\colon S\to B$ be a primitive cyclic covering fibration of type $(g,0,n)$.
Then we have 
$$
e_f(F_p)=(n-1)\alpha_0(F_p)+n\sum_{k\ge 1}\alpha_k(F_p)
-(2n-1)\varepsilon(F_p)
$$
for any $p \in B$.
\end{prop}

\begin{proof}
It is sufficient to show that 
$$
e_{\widetilde{f}}(\widetilde{F}_p)=
(n-1)\alpha^+_0(F_p)+n\sum_{k\ge 1}\alpha_k(F_p)-2(n-1)j(F_p).
$$ 
Let $N=\sum_{k\ge 1}\alpha_k(F_p)$ be the number of blow-ups on $\Gamma_p$ and $\widetilde{\Gamma}_p=\sum_{i=0}^{N}m_i \Gamma_i$ the irreducible decomposition. We may assume $\Gamma_i$ and $\widetilde{R}_h$ are transverse for simplicity. Put $r_i=\Gamma_i\widetilde{R}$ and $F_i=\widetilde{\theta}^{\ast} \Gamma_i$. Then, 
$$
F_p=\sum_{i=0}^N m_i F_i=\sum_{r_i>0} m_i F_i+\sum_{r_i=0} m_i(F_{i,1}+\cdots+F_{i,n})+\sum_{\Gamma_i\subset \widetilde{R}} m_i nF^{\prime}_i,
$$ 
where $F_{i,j}$, $F^{\prime}_i$ are smooth rational curves. For $r_i>0$, the restriction map $F_i\rightarrow \Gamma_i$ is an $n$-cyclic covering. From the Hurwitz formula, we have $2g(F_i)-2=-2n+(n-1)r_i$. Let $N_1$, $N_2$ and $N_3$ be the number of intersection points of two $\Gamma_i$ and $\Gamma_j$ which is contained in $\widetilde{R}_h$, not contained in $\widetilde{R}_h$ and that one $\Gamma_i$ is contained in $\widetilde{R}$, respectively. Clearly, it follows $N=N_1+N_2+N_3$. Let $J=j(F_p)$ and $K$ the number of $\Gamma_i$ such that $r_i=0$. Then, we have
\begin{eqnarray*}
e(\widetilde{F}_p)&=&\sum_{r_i>0} e(F_i)+2nK+2J-N_1-nN_2-N_3\\
&=&\sum_{r_i>0}(2n-(n-1)r_i)+2nK+2J-N_1-nN_2-N_3\\
&=&2n(N+1)-2(n-1)J-(n-1)\sum_{r_i>0} r_i-N-(n-1)N_2.
\end{eqnarray*}
Since $e(\widetilde{F})=2n-(n-1)r$, we have
\begin{eqnarray*}
e_{\widetilde{f}}(\widetilde{F}_p)=(2n-1)N-2(n-1)J+(n-1)\left(r-\sum_{r_i>0} r_i\right)-(n-1)N_2.
\end{eqnarray*}
On the other hand, we have
\begin{eqnarray*}
\alpha^+_0(F_p)&=&r-\#({\rm Supp}(\widetilde{\Gamma}_p)\cap {\rm Supp}(\widetilde{R}))\\
&=&r-\sum_{r_i>0} r_i+N_1+N_3.
\end{eqnarray*}
Combining these equalities, the assertion follows.
\end{proof}

\end{document}